\documentclass[12pt]{amsart}
\usepackage{amscd}
\usepackage{amssymb}
\usepackage[initials]{amsrefs}
\usepackage{enumitem}
\usepackage{ifpdf}
\ifpdf
\usepackage{xcolor}
\usepackage[arrow,tips,curve,frame,arc,pdf,color,line]{xy}
\definecolor{verylightgrey}{gray}{.8}
\definecolor{bordergrey}{gray}{.8}
\definecolor{lightgrey}{gray}{.1}
\else
\usepackage[arrow,tips,curve,frame,arc,ps,dvips,color,line]{xy}
\newxyColor{verylightgrey}{.8}{gray}{}
\newxyColor{bordergrey}{.8}{gray}{}
\newxyColor{lightgrey}{.1}{gray}{}
\fi
\SelectTips{lu}{12}
\UseTips
\newlength{\pinch}
\newlength{\mim}
\newcommand{\g}{6\mim}
\newcommand{\slice}[2]{%
\POS ="point";p+"c"="bot" **\dir{},?(.5)+/u4.5\pinch/="ccp",
"bot" **[#1]\crv{"ccp"},
p+"d"="top" **\dir{},?(.5)+/d9\pinch/="dcp",
"top";"bot"**\dir{},?(.5)*[#1]\xycircle<1.5\mim,\g>{}, 
"point";"top" **[#2]\crv{"dcp"}
}
\newcommand{\blownupslice}[2]{%
\POS ="lower";p+"c"="bot" **\dir{},?(.5)+/u4.5\pinch/="ccp",
"bot" **[#1]\crv{"ccp"},
p+<0\mim,30\mim>="upper" **[#2]\dir{-},
"lower"-<0\mim,\g>;"lower"**[#2]\dir{--},
"upper";"upper"+<0\mim,\g>**[#2]\dir{--},
"upper";p+"d"="top" **\dir{},?(.5)+/d9\pinch/="dcp",
"top" **[#2]\crv{"dcp"},
"top";"bot"**[#1]\dir{-},
"bot"-<0\mim,\g>;"bot"**[#1]\dir{--},
"top";"top"+<0\mim,\g>**[#2]\dir{--}
}

\DeclareSymbolFont{pssymbols}     {OMS}{ztmcm}{m}{n}
\DeclareSymbolFontAlphabet{\mathpsscr}   {pssymbols}

\theoremstyle{plain}%
\newtheorem{thm}{Theorem}[section]
\newtheorem{prop}[thm]{Proposition}
\newtheorem{lem}[thm]{Lemma}
\newtheorem{cor}[thm]{Corollary}

\theoremstyle{definition}%
\newtheorem{defi}[thm]{Definition}

\theoremstyle{remark}%
\newtheorem{rem}[thm]{Remark}

\newcommand{\oX}{\overline{X}}
\newcommand{\osp}{\oX}
\newcommand{\ga}{\Gamma}
\newcommand{\olsp}{\ga\backslash\osp}
\newcommand{\olX}{\ga\backslash\oX}
\newcommand{\G}{\mathbf G}

\renewcommand{\P}{{\mathbf P}}
\newcommand{\n}{\mathbf N}
\DeclareMathOperator{\SL}{\mathbf S\mathbf L}
\DeclareMathOperator{\GL}{\mathbf G\mathbf L}

\DeclareMathOperator{\CAT}{CAT}

\newcommand{\rank}{\text{-rank}\:}
\DeclareMathOperator{\Isom}{Isom}

\DeclareMathOperator{\Hom}{Hom}
\DeclareMathOperator{\Aff}{Aff}
\newcommand{\aff}{_{\mathrm{af}}}

\newcommand{\R}{{\mathbb R}}
\newcommand{\Q}{{\mathbb Q}}
\newcommand{\C}{{\mathbb C}}
\newcommand{\Z}{{\mathbb Z}}
\newcommand{\N}{{\mathbf N}}   
\newcommand{\Split}{{\mathbf S}}  
\newcommand{\Cent}{{\mathbf Z}}   

\newcommand{\La}{{\mathfrak a}}
\newcommand{\Lb}{{\mathfrak b}}

\newcommand{\OO}{{\mathcal O}}

\renewcommand{\max}{{\textup{max}}}

\begin{document}

\title[Fundamental Group]{The
  Fundamental Group of Reductive Borel-Serre and Satake Compactifications}
\author{Lizhen Ji}
\address{Department of Mathematics\\University of Michigan\\Ann Arbor, MI
  48109}
\email{lji@umich.edu}
\author{V. Kumar Murty}
\address{Department of Mathematics\\University of Toronto\\40 St.\ George
  Street\\Toronto, Ontario, M5S 3G3\\CANADA}
\email{murty@math.toronto.edu}
\author{Leslie Saper}
\address{Department of Mathematics\\Duke University\\Durham, NC 27708}
\email{saper@math.duke.edu}
\author{John Scherk}
\address{Department of Computer and Mathematical Sciences\\University of
  Toronto Scarborough\\1265 Military Trail
\\Toronto, Ontario, M1C 1A4\\CANADA}
\email{scherk@math.toronto.edu}
\date{March 6, 2014}
\thanks{Research partially supported by NSF and NSERC grants}

\begin{abstract}
Let $\G$ be an almost simple, simply connected algebraic group defined over
a number field $k$, and let $S$ be a finite set of places of $k$ including
all infinite places.  Let $X$ be the product over $v\in S$ of the symmetric
spaces associated to $\G(k_v)$, when $v$ is an infinite place, and the
Bruhat-Tits buildings associated to $\G(k_v)$, when $v$ is a finite place.
The main result of this paper is to compute explicitly the fundamental group of
the reductive Borel-Serre compactification of $\ga\backslash X$,
where $\ga$ is an $S$-arithmetic subgroup of $\mathbf G$.  In the case that $\ga$ is neat,
we show that this
fundamental group is isomorphic to $\ga/E\ga$, where $E\ga$ is the subgroup
generated by the elements of $\ga$ belonging to unipotent radicals of
$k$-parabolic subgroups.  Analogous computations of the fundamental group of
the Satake compactifications are made. It is noteworthy that calculations of the congruence
subgroup kernel $C(S,\G)$ yield similar results.
\end{abstract}
\maketitle

\section{Introduction}

Let $X$ be a symmetric space of noncompact type, and let
$G = \Isom(X)^0$.
Pick a basepoint $x_0 \in X$, with isotropy group $K \subset G$.
Then $X \cong G/K$.
Suppose that $G=\G(\R)^0$, where $\G$ is a connected almost simple
algebraic group defined over $\Q$,
and that $\Gamma$ is an arithmetic subgroup of $G$. The associated locally
symmetric space $\Gamma\backslash X$ is typically not compact.
Noncompact arithmetic locally symmetric spaces $\Gamma\backslash X$
admit several different compactifications such
as the Satake compactifications, the Borel-Serre compactification and the
reductive Borel-Serre compactification.  If $\Gamma\backslash X$ is a
Hermitian locally symmetric space, then one of the Satake
compactifications, called the Baily-Borel or Baily-Borel Satake
compactification, is a projective variety.

The cohomology and homology groups of locally symmetric spaces
$\Gamma\backslash X$ and of their compactifications have been intensively
studied because of their relation to the cohomology of $\Gamma$ and to
automorphic forms.
Our interest here is in the fundamental group.  There are a number of results on the
fundamental group of the Baily-Borel Satake compactification of particular
Hermitian locally symmetric spaces  (see \citelist{\cite{hk} \cite{kn} \cite{hs}
  \cite{ge} \cite{Gro} \cite{gro2} \cite{san}}).  In this paper we deal with
arithmetic locally symmetric spaces in general (not necessarily Hermitian).
A natural compactification of $\Gamma\backslash X$ to consider in this context
is the reductive Borel-Serre compactification, which more and more is playing a central role
(see, for example \citelist{\cite{AyoubZucker}\cite{JiMacPherson}\cite{sap3}}).
We show that its fundamental group can be described in terms of ``elementary matrices''.
We also determine the fundamental
group of arbitrary Satake compactifications of a locally symmetric space and actually treat
the more general case of $\Gamma$ an $S$-arithmetic group.

In the remainder of this section we state our results more precisely. In sections 2 and 3 we recall results about
compactifications of locally symmetric spaces and Bruhat-Tits buildings, which we will
need in our computations of fundamental groups. Of particular importance are
descriptions of the topology of these spaces and of the stabilizers of points under the
built-in group action. In section 4 we define the reductive Borel-Serre and Satake compactifications 
of $\Gamma\backslash X$ for a general $S$-arithmetic group $\Gamma$.
In sections 5 and 6 we prove our results.

\subsection{$S$-arithmetic groups and elementary matrices}

Let $k$ be a number field and let $S$ be a finite set of places of $k$ which
contains the infinite places $S_\infty$.  Set $S_f = S \setminus S_\infty$.
For $v\in S$ let $k_v$ denote the
completion of $k$ with respect to a norm associated to $v$.  Denote by
$\OO$ the ring of \emph{$S$-integers}
$$ 
\OO = \{\, x\in k \mid \operatorname{ord}_v (x) \ge 0 \text{ for all
  $v\notin   S$}\,\} \ .
$$
The corresponding group of units $\OO^\times$
is finite if and only if $|S|=1$.

Let $\G$ be an algebraic group defined over $k$ and fix a faithful
representation
\begin{equation*}
\rho\colon \G \longrightarrow \GL_N
\end{equation*}
defined over $k$. Set
\begin{equation*}
\G(\OO) = \rho^{-1}(\GL_N(\OO)) \subset \G(k) .
\end{equation*}
Note that $\G(\OO)$ depends on the representation $\rho$.

A subgroup $\ga \subset \G(k)$ is an \emph{$S$-arithmetic subgroup} if it
is commensurable with $\G(\OO)$; an $S_\infty$-arithmetic subgroup is
simply called an \emph{arithmetic subgroup}.  This definition is
independent of the choice of $\rho$. Note that if $S_1 \subset S_2$, then
an $S_1$-arithmetic group is not necessarily an $S_2$-arithmetic
group. (For example, $\SL_2(\Z)$ is of infinite index in $\SL_2(\Z[1/p])$
for any prime $p$ and in particular they are not commensurable.)  However
if $\ga$ is an $S_2$-arithmetic subgroup and $K_v$ is a compact open
subgroup of $\G(k_v)$ for each $v\in S_2\setminus S_1$, then $\ga\cap
\bigcap_{v\in S_2\setminus S_1} K_v$ is an $S_1$-arithmetic subgroup.

For any $S$-arithmetic subgroup
$\ga$ let
\begin{equation*}
E\ga \subset \ga
\end{equation*}
be the subgroup generated by the elements of $\ga$ belonging to the
unipotent radical of any parabolic $k$-subgroup of $\G$ (the subgroup of
 ``elementary matrices''). Let 
 $$
 S\rank \G = \sum_{v\in S} k_v\rank \G\ .
 $$
 If $k\rank \G >0 $ and $S\rank \G \ge 2$, then $E\ga$
 is $S$-arithmetic \citelist{\cite{Margulis} \cite{Ra2}*{Theorem~ A, Corollary~ 1}}.
 
\subsection{Fundamental groups}

Now let $\G$ be connected, absolutely almost simple, and
simply connected. Let $\mathbf H$
denote the restriction of scalars $\operatorname{Res}_{k/\Q} \G$ of $\G$;
this is a group defined over $\Q$ with $\Q\rank \mathbf H = k\rank \G$.
Let $X_\infty =\mathbf H(\R)/K$ be the symmetric space associated to
$\mathbf H$, where $K$ is a maximal compact subgroup of $\mathbf H(\R)$,
and for $v\in S_f $, let $X_v$ be the Bruhat-Tits building
of $\G(k_v)$.

Consider $X = X_\infty \times \prod_{v\in S\setminus S_\infty} X_v$.  By
extending the work of Borel and Serre \citelist{\cite{Borel-Serre}
\cite{BS2}} and of Zucker \cite{Zu1}, we define in
\S\S\ref{subsectRBSarith}, \ref{subsectRBSSarith} the reductive Borel-Serre
bordification $\overline{X}^{RBS}$ of $X$.  For an $S$-arithmetic
subgroup $\ga$ of $\G(k)$, the action of $\ga$ on $X$ by left translation
extends to $\overline{X}^{RBS}$ and the quotient
$\ga\backslash\overline{X}^{RBS}$ is a compact Hausdorff topological
space, called the \emph{reductive Borel-Serre compactification} of
$\ga\backslash X$.  Our main result (Theorem ~\ref{thmMainArithmetic}) is the
computation of the fundamental group of
$\ga\backslash\overline{X}^{RBS}$.  Under the mild condition that $\ga$
is a neat $S$-arithmetic group, we show (Corollary ~\ref{corNeat}) that
\begin{equation}
\label{eqnFundGp}
\pi_1(\ga\backslash\overline{X}^{RBS}) \cong  \ga / E\ga\ .
\end{equation}
If  $k\rank \G >0$ and $S\rank \G \ge 2$ this is finite. The Satake
compactifications of the locally symmetric space $\ga\backslash X_\infty$
are important as well, as mentioned at the beginning of this introduction.
In \S\ref{subsectSatakeSArith} we define compactifications $\ga\backslash
{}_{\Q}\overline{X}^{\tau}$ of $\ga \backslash X$ which generalize the
Satake compactifications of $\ga\backslash X_\infty$ and in
\S\ref{sectFundGrpArithmetic} we calculate that their fundamental groups
are a certain quotient of $\pi_1(\ga\backslash\overline{X}^{RBS})$.

\begin{rem}
There is an intriguing similarity between our results on the fundamental group 
and computations of the congruence subgroup kernel $C(S,\G)$ (see \cite{SerreBourbaki}
and  \cite{Ra1}). For more details, see the Appendix.
\end{rem}

\section{The reductive Borel-Serre and Satake compactifications: the
  arithmetic case}
\label{sectCompactificationsArithmetic}

In order to establish notation and set the framework for later proofs, we
recall in \S\S\ref{ssectBSarith}--\ref{subsectSatakeArith} several natural
compactifications of the locally symmetric space $\ga\backslash X_\infty$
associated to an arithmetic group $\ga$; in each case a bordification of
$X_\infty$ is described on which $\G(k)$ acts.  We also examine the
stabilizer subgroups of points in these bordifications.  The case of
general $S$-arithmetic groups will be treated in
\S\ref{sectCompactificationsSArithmetic}.  Throughout the paper, $\G$ will
denote a connected, absolutely almost simple, simply connected algebraic
group defined over a number field $k$.

\subsection{Proper and discontinuous actions}
\label{ssectProperDiscontinuousActions}
Recall \cite{BourbakiTopologiePartOne}*{III, \S4.4, Prop.~7} that a
discrete group $\ga$ acts \emph{properly} on a Hausdorff space $Y$ if and
only if for all $y$, $y'\in Y$, there exist neighborhoods $V$ of $y$ and
$V'$ of $y'$ such that $\gamma V\cap V'\neq \emptyset$ for only finitely
many $\gamma \in \ga$.  We will also need the following weaker condition on
the group action:

\begin{defi}[\cite{Gro}*{Definition~1}]
\label{defnDiscontinuous}
The action of a discrete group $\ga$ on a topological space $Y$ is
\emph{discontinuous} if
\begin{enumerate}
\item\label{itemDiscontinuousTwoPoints} for all $y$, $y'\in Y$ with
  $y'\notin \ga y$ there exists neighborhoods $V$ of $y$ and $V'$ of $y'$
  such that $\gamma V\cap V' =\emptyset$ for all $\gamma\in \ga$, and
\item\label{itemDiscontinuousOnePoint} for all $y\in Y$ there exists a
  neighborhood $V$ of $y$ such that $\gamma V\cap V = \emptyset$ for
  $\gamma \notin \ga_y$ and $\gamma V = V$ for $\gamma \in \ga_y$.
\end{enumerate}
\end{defi}

It is easy to check that a group action is proper if and only if it is
discontinuous and the stabilizer subgroup $\ga_y$ is finite for all $y\in
Y$.

\begin{defi}
The \emph{stabilizer} $\Gamma^X$ of a subset $X \subset Y$ is the subgroup
$$
\Gamma^X = \{ \gamma\,  \vert\,  \gamma X = X \}\ .
$$
The \emph{fixing} group (\emph{fixateur}) $\Gamma_X$  is
$$
\Gamma_X = \{ \gamma\,  \vert\,  \gamma x = x,\ \text{for all}\ x \in X \}\ .
$$
Thus for $y \in Y$, $\Gamma^y = \Gamma_y$. Lastly, let
 $\ga_{fix}$ be the subgroup generated by the
  stabilizer subgroups $\ga_{y}$ for all $y\in Y$.  (This subgroup is
  obviously normal.)

\end{defi}

\subsection{The locally symmetric space associated to an arithmetic subgroup}
Let $S_{\infty}$ be the set of all
infinite places of $k$.  For each $v\in S_\infty$, let $k_{v}$ be the
corresponding completion of $k$ with respect to a norm associated with $v$;
thus either $k_{v}\cong \R$ or $k_{v}\cong \C$.  For each $v\in
S_{\infty}$, $\G(k_{v})$ is a (real) Lie group.

Define $G_{\infty}=\prod_{v\in S_{\infty}}\G(k_{v})$, a semisimple Lie
group with finitely many connected components.  Fix a maximal compact
subgroup $K$ of $G_{\infty}$.  When endowed with a $G$-invariant metric,
$X_\infty = G_{\infty}/K$ is a Riemannian symmetric space of noncompact
type and is thus contractible.  Embed $\G(k)$ into $G_\infty$ diagonally.
Then any arithmetic subgroup $\ga\subset \G(k)$ is a discrete subgroup of
$G_\infty$ and acts properly on $X_\infty$.  It is known that the quotient
$\ga\backslash X_\infty$ is compact if and only if the $k$-rank of $\G$ is
equal to 0.  In the following, we assume that the $k$-rank of $\G$ is
positive so that $\ga\backslash X_\infty$ is noncompact.

Since the theory of compactifications of locally symmetric spaces is
usually expressed in terms of algebraic groups defined over $\Q$, let
$\mathbf H=\operatorname{Res}_{k/\Q}\G$ be the algebraic group defined over
$\Q$ obtained by restriction of scalars; it satisfies
\begin{equation}
\label{eqnPointsOfRestrictionScalars}
\mathbf H(\Q)=\G(k) \quad\text{and}\quad \mathbf H(\mathbb R)=G_{\infty}\ .
\end{equation}
The space $X_{\infty}$ can be identified with the symmetric space of
maximal compact subgroups of $\mathbf H(\mathbb R)$, $X_{\infty}=\mathbf
H(\mathbb R)/K$, and the arithmetic subgroup $\ga\subset \G(k)$ corresponds
to an arithmetic subgroup $\ga\subset \mathbf H(\Q)$.  Restriction of
scalars yields a one-to-one correspondence between parabolic $k$-subgroups
of $\G$ and parabolic $\Q$-subgroups of $\mathbf H$ so that the analogue of
\eqref{eqnPointsOfRestrictionScalars} is satisfied.

\subsection{The Borel-Serre compactification}
\label{ssectBSarith}
(For details see the original paper \cite{Borel-Serre}, as well as
\cite{Borel-Ji}.)  For each parabolic $\Q$-subgroup $\P$ of $\mathbf H$,
consider the Levi quotient $\mathbf L_{\P} = \P/\n_{\P}$ where $\n_{\P}$ is
the unipotent radical of $\P$.  This is a reductive group defined over
$\Q$.  There is an almost direct product $\mathbf L_{\P} = \mathbf S_{\P}
\cdot \mathbf M_{\P}$, where $\mathbf S_{\P}$ is the maximal $\Q$-split
torus in the center of $\mathbf L_{\P}$ and $\mathbf M_{\P}$ is the
intersection of the kernels of the squares of all characters of $\mathbf
L_{\P}$ defined over $\Q$.  The real locus $L_P= \mathbf L_\P(\R)$ has a
direct product decomposition $A_P \cdot M_P$, where $A_P = \mathbf
S_\P(\R)^0$ and $M_P = \mathbf M_\P(\R)$.  The dimension of $A_P$ is called
the \emph{parabolic $\Q$-rank} of $\P$.

The real locus $P=\P(\R)$ has a Langlands decomposition
\begin{equation}\label{rationalLanglands}
P=N_{P} \ltimes (\widetilde A_P \cdot \widetilde M_ P),
\end{equation}
where $N_{P}= \n_{\P}(\R)$ and $\widetilde A_P \cdot \widetilde M_ P$ is
the lift of $A_P \cdot M_P$ to the unique Levi subgroup of $P$ which is
stable under the Cartan involution $\theta$ associated with $K$.

Since $P$ acts transitively on $X_\infty$, the Langlands decomposition induces a
horospherical decomposition
\begin{equation}\label{horo}
X_\infty \cong  A_P\times N_{P}\times X_P,\quad u\tilde a\tilde mK \mapsto
(\tilde a,u,\tilde m(K\cap \widetilde M_P),
\end{equation}
where 
\begin{equation*}
X_P= \widetilde M_P / (K \cap \widetilde M_P) \cong  L_P/(A_P\cdot K_P)
\end{equation*}
is a symmetric space (which might contain an Euclidean factor) and is
called the \emph{boundary symmetric space associated with $\P$}.  The
second expression for $X_P$ is preferred since $\mathbf L_\P$ is defined
over $\Q$; here $K_P\subseteq \mathbf M_\P(\R)$ corresponds to $K \cap
\widetilde M_P$

For each parabolic $\Q$-subgroup $\P$ of $\mathbf H$, define the
Borel-Serre boundary component
\begin{equation*}
e(P)=N_{P}\times X_P
\end{equation*}
which we view as the quotient of $X_\infty$ obtained by collapsing the
first factor in \eqref{horo}.  The action of $P$ on $X_\infty$ descends to
an action on $e(P)=N_{P}\times X_P$ given by
\begin{equation}
\label{eqnBoundaryAction}
p\cdot (u, y) = (pu\tilde m_p^{-1}\tilde a_p^{-1} , \tilde a_p \tilde m_p
y), \qquad \text{for $p=u_p \tilde a_p \tilde m_p\in P$.}
\end{equation}
Define the Borel-Serre partial compactification $\overline{X}_\infty^{BS}$ (as a
\emph{set}) by
\begin{equation}
\label{BSPartialCompactification}
\overline{X}_\infty^{BS}=X_\infty\cup \coprod_{\P\subset \mathbf H} e(P).
\end{equation}

Let $\Delta_P$ be the simple ``roots'' of the adjoint action of $A_P$ on
the Lie algebra of $N_P$ and identify $A_P$ with $(\R^{>0})^{\Delta_P}$ by
$a \mapsto (a^{-\alpha})_{\alpha\in\Delta_P}$.  Enlarge $A_P$ to the
topological semigroup $\overline A_P \cong (\R^{\ge0})^{\Delta_P}$ by
allowing $a^\alpha$ to attain infinity and define
\begin{equation*}
\overline A_P(s) = \{\, a\in \overline A_P\mid a^{-\alpha} < s^{-1} \text{
  for all $\alpha\in \Delta_P$}\,\}\cong [0,s^{-1})^{\Delta_P}\ ,\qquad
  \text{for $s>0$}\ .
\end{equation*}
Similarly enlarge the Lie algebra $\La_P \subset \overline\La_P$.  The
inverse isomorphisms $\exp\colon \La_P \to A_P$ and $\log\colon A_P \to
\La_P$ extend to isomorphisms
\begin{equation*}
\overline A_P \overset{\log}{\longrightarrow} \overline \La_P
\qquad\text{and} \qquad \overline \La_P \overset{\exp}{\longrightarrow}
\overline A_P.
\end{equation*}

To every parabolic $\Q$-subgroup $\mathbf Q\supseteq \P$ there corresponds
a subset $\Delta_P^Q \subseteq \Delta_P$ and we let $o_Q\in
\overline A_P$ be the point with coordinates $o_Q^{-\alpha} =1$
for $\alpha\in \Delta_P^Q$ and $o_Q^{-\alpha} =0$ for
$\alpha\notin \Delta_P^Q$.  Then $\overline A_P = \coprod_{\mathbf Q
\supseteq \P} A_P \cdot o_Q$ is the decomposition into
$A_P$-orbits.

Define the \emph{corner associated to $\mathbf P$} to be
\begin{equation}
\label{Pcorner}
X_\infty(P) = \overline A_P \times e(P) = \overline A_P \times N_P \times X_P.
\end{equation}
We identify $e(Q)$ with the subset $ (A_P\cdot o_Q) \times N_P\times X_P$.
In particular, $e(P)$ is identified with the subset $\{o_P\}\times
N_P\times X_P$ and $X_\infty$ is identified with the open subset $A_P \times
N_P\times X_P \subset X_\infty(P)$ (compare \eqref{horo}).  Thus we have a
bijection
\begin{equation}
\label{strataPcorner}
X_\infty(P) \cong X_\infty \cup \coprod_{\P \subseteq \mathbf Q \subset
  \mathbf H} e(Q).
\end{equation}

Now give $\overline X_\infty^{BS}$ the finest topology so that for all
parabolic $\Q$-subgroups $\P$ of $\mathbf H$ the inclusion of
\eqref{strataPcorner} into \eqref{BSPartialCompactification} is a
continuous inclusion of an open subset.  Under this topology, a sequence
$x_n\in X$ converges in $\overline X_\infty^{BS}$ if and only if there
exists a parabolic $\Q$-subgroup $\P$ such that if we write $x_n=(a_n, u_n,
y_n)$ according to the decomposition of \eqref{horo}, then $(u_n,y_n)$
converges to a point in $e(P)$ and $a_n^\alpha\to \infty$ for all
$\alpha\in \Delta_P$.  The space $\overline X_\infty^{BS}$ is a manifold
with corners.  It has the same homotopy type as $X_\infty$ and is thus
contractible \cite{Borel-Serre}.

The action of $\mathbf H(\Q)$ on $X_\infty$ extends to a continuous action
on $\overline{X}_\infty^{BS}$ which permutes the boundary components:
$g\cdot e(P) = e(gPg^{-1})$ for $g\in \mathbf H(\Q)$.  The normalizer of
$e(P)$ is $\P(\Q)$ which acts according to \eqref{eqnBoundaryAction}.

It is shown in \cite{Borel-Serre} that the action of $\ga$ on
$\overline{X}_\infty^{BS}$ is proper and the quotient $\ga\backslash
\overline{X}_\infty^{BS}$, the \emph{Borel-Serre compactification}, is a compact
Hausdorff space.  It is a manifold with corners if $\ga$ is torsion-free.

\subsection{The reductive Borel-Serre compactification}
\label{subsectRBSarith}
This compactification was first constructed by Zucker \cite{Zu1}*{\S4} (see also
\cite{GHM}).  For each parabolic $\Q$-subgroup $\P$ of $\mathbf H$, define
its reductive Borel-Serre boundary component $\hat{e}(P)$ by
\begin{equation*}
\hat{e}(P)=X_P
\end{equation*}
and set
\begin{equation*}
\overline{X}_\infty^{RBS}=X_\infty\cup \coprod_{\P} \hat{e}(P).
\end{equation*}
The projections $p_P\colon e(P) = N_P\times X_P \to \hat e(P) = X_P$ induce
a surjection $p\colon \overline{X}_\infty^{BS} \to
\overline{X}_\infty^{RBS}$ and we give $\overline{X}_\infty^{RBS}$ the
quotient topology.  Its topology can also be described in terms of
convergence of interior points to the boundary points via the horospherical
decomposition in equation \eqref{horo}.  Note that $\overline{X}^{RBS}$ is
not locally compact, although it is compactly generated (being a Hausdorff
quotient of the locally compact space $\overline{X}^{BS}$).  The action of
$\mathbf H(\Q)$ on $\overline{X}_\infty^{BS}$ descends to a continuous
action on $\overline{X}_\infty^{RBS}$.

\begin{lem}
\label{lemStabilizersRBS}
Let $\P$ be a parabolic $\Q$-subgroup of $\mathbf H$.
The stabilizer $\mathbf H(\Q)_z= \G(k)_z$ of $z\in X_P$ under the action of
$\mathbf H(\Q)$ on $\oX^{RBS}_\infty$ satisfies a short exact sequence
\begin{equation*}
1 \to \n_{\P}(\Q) \to \mathbf H(\Q)_z \to \mathbf L_{\P}(\Q)_z \to 1
\end{equation*}
where $\mathbf L_{\P}(\Q)_z$ is the stabilizer of $z$ under the action of
$\mathbf L_{\P}(\Q)$ on $X_P$.
\end{lem}
\begin{proof}
The normalizer of $X_P$ under the action of $\mathbf H(\Q)$ is $\P(\Q)$
which acts via its quotient $\mathbf L_{\P}(\Q)$. 
\end{proof}

By the lemma, the action of $\ga$ on $\overline{X}_\infty^{RBS}$ is not
proper since the stabilizer of a boundary point in $X_P$ contains the
infinite group $\ga_{N_P} = \ga\cap N_P$.  Nonetheless
\begin{lem}
\label{lemRBSDiscontinuous}
The action of an arithmetic subgroup $\ga$ on $\overline{X}_\infty^{RBS}$
is discontinuous and the arithmetic quotient
$\ga\backslash\overline{X}_\infty^{RBS}$ is a compact Hausdorff space.
\end{lem}

\begin{proof}
We begin by verifying Definition
~\ref{defnDiscontinuous}\ref{itemDiscontinuousOnePoint}.  Let $x\in X_P
\subseteq \overline{X}_\infty^{RBS}$.  Set $\ga_P = \ga\cap P$ and
$\ga_{L_P} = \ga_P/\ga_{N_P}$.  Since $\ga_{L_P}$ acts properly on $X_P$
there exists a neighborhood $O_x$ of $x$ in $X_P$ such that $\bar \gamma
O_x \cap O_x \neq \emptyset$ if and only if $\bar \gamma \in \ga_{L_P,x}$,
in which case $\bar \gamma O_x = O_x$.  We can assume $O_x$ is relatively
compact.  Set $V=p(\overline{A}_P(s)\times N_P \times O_x)$, where we chose
$s$ sufficiently large so that that only identifications induced by $\ga$
on $V$ already arise from $\ga_P$ \cite{Zu3}*{(1.5)}.  Thus $\gamma V\cap
V\neq \emptyset$ if and only if $\gamma \in \ga_P$ and $\gamma \ga_{N_P}
\in \ga_{L_P,x}$; by Lemma ~\ref{lemStabilizersRBS} this occurs if and only
if $\gamma \in \ga_x$ as desired.

To verify Definition
~\ref{defnDiscontinuous}\ref{itemDiscontinuousTwoPoints} we will show the
equivalent condition that $\ga\backslash\overline{X}_\infty^{RBS}$ is
Hausdorff (compare \cite{Zu1}*{(4.2)}).  Compactness will follow since it
is the image of a compact space under the induced projection $p'\colon
\ga\backslash \overline{X}_\infty^{BS} \to \ga\backslash
\overline{X}_\infty^{RBS}$.  Observe that $p'$ is a quotient map and that
its fibers, each being homeomorphic to $\ga_{N_P}\backslash N_P$ for some
$\P$, are compact. For $y\in \ga\backslash\overline{X}_\infty^{RBS}$ and
$W$ a neighborhood of $p'^{-1}(y)$, we claim there exists $U\ni y$ open
such that $p'^{-1}(U)\subseteq W$.  This suffices to establish Hausdorff,
for if $y_1\neq y_2 \in \ga\backslash\overline{X}_\infty^{RBS}$ and $W_1$
and $W_2$ are disjoint neighborhoods of the compact fibers $p'^{-1}(y_1)$
and $p'^{-1}(y_2)$, there must exist $U_1$ and $U_2$, neighborhoods of
$y_1$ and $y_2$, such that $p'^{-1}(U_i) \subseteq W_i$ and hence $U_1\cap
U_2 =\emptyset$.

To prove the claim, choose $x\in X_P$ such that $y=\ga x$.  Let $q\colon
\overline{X}_\infty^{BS} \to \ga\backslash \overline{X}_\infty^{BS} $ be
the quotient map.  The compact fiber $p'^{-1}(y)$ may be covered by
finitely many open subsets $q(\overline A_P(s_\mu)\times C_{P,\mu} \times
O_{P,\mu}) \subseteq W$ where $C_{P,\mu} \subseteq N_P$ and $x\in
O_{P,\mu}\subseteq X_P$.  Define a neighborhood $V$ of the fiber by
\begin{equation*}
p'^{-1}(y) \subset V = q(\overline A_P(s)\times C_{P}
\times O_{P}) \subseteq W
\end{equation*}
where $s = \max \,s_\mu$, $O_P = \bigcap O_{P,\mu}$, and $C_P = \bigcup C_{P,\mu}$.
Since $\ga_{N_P}C_P = N_P$, we see $V=p'^{-1}(U)$ for some $U\ni y$ as
desired.
\end{proof}

\subsection{Satake compactifications}
\label{subsectSatakeArith}
For arithmetic quotients of $X_\infty$, the Satake compactifications
$\ga\backslash {}_{\Q}\overline{X}_\infty^{\tau}$ form an important family of
compactifications.  When $X_\infty$ is Hermitian, one example is the Baily-Borel
Satake compactification. The construction has three steps.
\begin{enumerate}
\item Begin%
\footnote{Here we follow \cite{Cass} in beginning with a spherical
  representation. Satake's original construction \cite{sat1} started with a
  non-spherical representation but then constructed a spherical
  representation by letting $G_\infty$ act on the space of self-adjoint
  endomorphisms of $V$ with respect to an admissible inner product. See
  \cite{sap2} for the relation of the two constructions.}
with a representation $(\tau,V)$ of $\mathbf H$ which has a nonzero
$K$-fixed vector $v\in V$ (a \emph{spherical representation}) and which is
irreducible and nontrivial on each noncompact $\R$-simple factor of
$\mathbf H$.  Define the Satake compactification $\overline{X}_\infty^{\tau}$
of $X$ to be the closure of the image of the embedding $X_\infty \hookrightarrow
\mathbb P(V)$, $gK \mapsto [ \tau(g) v]$.  The action of $G_\infty$ extends
to a continuous action on $\overline{X}_\infty^{\tau}$ and the set of points
fixed by $N_P$, where $\P$ is any parabolic $\R$-subgroup, is called a
\emph{real boundary component}.  The compactification
$\overline{X}_\infty^{\tau}$ is the disjoint union of its real boundary
components.

\item Define a partial compactification
  ${}_{\Q}\overline{X}_\infty^{\tau}\subseteq \overline{X}_\infty^{\tau}$
  by taking the union of $X_\infty$ and those real boundary components that
  meet the closure of a Siegel set.  Under the condition that
  $\overline{X}_\infty^{\tau}$ is \emph{geometrically rational}
  \cite{Cass}, this is equivalent to considering those real boundary
  components whose normalizers are parabolic $\Q$-subgroups; call these the
  \emph{rational boundary components}.  Instead of the subspace topology
  induced from $\overline{X}_\infty^{\tau}$, give
  ${}_{\Q}\overline{X}_\infty^{\tau}$ the Satake topology \cite{sat2}.

\item Still under the condition that $\overline{X}_\infty^{\tau}$ is
  geometrically rational, one may show that the arithmetic subgroup $\ga$
  acts continuously on ${}_{\Q}\overline{X}_\infty^{\tau}$ with a
  compact Hausdorff quotient, $\ga\backslash
  {}_{\Q}\overline{X}_\infty^{\tau}$.  This is the \emph{Satake
  compactification} of $\ga\backslash X_\infty$.
\end{enumerate}

The geometric rationality condition above always holds if the
representation $(\tau,V)$ is rational over $\Q$ \cite{sap2}.  It also holds
for the Baily-Borel Satake compactification \cite{BB}, as well as most
equal-rank Satake compactifications including all those where $\Q\rank
\mathbf H >2$.

We will now describe an alternate construction of
${}_{\Q}\overline{X}_\infty^{\tau}$ due to Zucker \cite{Zu2}.  Instead of
the Satake topology, Zucker gives ${}_{\Q}\overline{X}_\infty^{\tau}$ the
quotient topology under a certain surjection $\overline{X}_\infty^{RBS}
\to {}_{\Q}\overline{X}_\infty^{\tau}$ described below.  It is this
topology we will use in this paper.  Zucker proves that the resulting two
topologies on $\ga\backslash {}_{\Q}\overline{X}_\infty^{\tau}$ coincide.

Let $(\tau,V)$ be a spherical representation as above.  We assume that
$\overline{X}_\infty^{\tau}$ is geometrically rational.  For any parabolic
$\Q$-subgroup $\P$ of $\mathbf H$, let $X_{P,\tau}\subseteq \overline
X_\infty^{\tau}$ be the real boundary component fixed pointwise by $N_P$;
geometric rationality implies that $X_{P,\tau}$ is actually a rational
boundary component.  The transitive action of $P$ on $X_{P,\tau}$ descends
to an action of $L_P = P/N_P$.  The geometric rationality condition ensures
that there exists a normal $\Q$-subgroup $\mathbf L_{\P, \tau} \subseteq
\mathbf L_{\P}$ with the property that $L_{P,\tau}= \mathbf L_{\P,
  \tau}(\R)$ is contained in the centralizer
$\operatorname{Cent}(X_{P,\tau})$ of $X_{P,\tau}$ and
$\operatorname{Cent}(X_{P,\tau})/L_{P,\tau}$ is compact.  Then $X_{P,\tau}$
is the symmetric space associated to the $\Q$-group $\mathbf H_{\P,
\tau} = \mathbf L_{\P} / \mathbf L_{\P,\tau}$.   There is an
almost direct product decomposition
\begin{equation}
\label{eqnSatakeLeviDecomposition}
\mathbf L_{\P} = \widetilde {\mathbf H}_{\P, \tau} \cdot \mathbf L_{\P,
  \tau}\ ,
\end{equation}
where $\widetilde {\mathbf H}_{\P, \tau}$ is a lift of $\mathbf H_{\P,
  \tau}$; the root systems of these factors may be described using the
highest weight of $\tau$.  We obtain a decomposition of symmetric spaces
\begin{equation}
\label{eqnBoundaryDecomposition}
X_P = X_{P,\tau} \times  W_{P,\tau}\ .
\end{equation}

Different parabolic $\Q$-subgroups can yield the same rational boundary
component $X_{P,\tau}$; if $\P^\dag$ is the maximal such parabolic
$\Q$-subgroup, then 
\begin{equation}\label{Pdagger}
P^\dag=\P^\dag(\R)
\end{equation}
is the normalizer of $X_{P,\tau}$.
The parabolic $\Q$-subgroups that arise as the normalizers of rational
boundary components are called \emph{$\tau$-saturated}.  For example, all
parabolic $\Q$-subgroups are saturated for the maximal Satake
compactification, while only the maximal parabolic $\Q$-subgroups are
saturated for the Baily-Borel Satake compactification when $\mathbf H$ is
$\Q$-simple.  In general, the class of $\tau$-saturated parabolic
$\Q$-subgroups can be described in terms of the highest weight of $\tau$.

Define 
\begin{equation*}
{}_{\Q}\overline{X}_\infty^{\tau}=X_\infty\cup \coprod_{\text{$\mathbf Q$
    $\tau$-saturated}} X_{Q,\tau}\ .
\end{equation*}
A surjection $p\colon \overline{X}_\infty^{RBS} \to
{}_{\Q}\overline{X}_\infty^{\tau}$ is obtained by mapping $X_P$ to
$X_{P,\tau} = X_{P^\dag,\tau}$ via the projection on the first factor in
\eqref{eqnBoundaryDecomposition}.  Give ${}_{\Q}\overline{X}_\infty^{\tau}$
the resulting quotient topology; the action of $\mathbf H(\Q)$ on
$\overline{X}_\infty^{RBS}$ descends to a continuous action on
${}_{\Q}\overline{X}_\infty^{\tau}$. 

Let $\P_\tau$ be the inverse image of $\mathbf L_{\P,\tau}$ under
the projection $\P \to \P/\n_{\P}$.
\begin{lem}
\label{lemStabilizersSatake}
Let $\P$ be a $\tau$-saturated parabolic $\Q$-subgroup of $\mathbf H$.  The
stabilizer $\mathbf H(\Q)_z = \G(k)_z$ of $z\in X_{P,\tau}$ under the
action of $\mathbf H(\Q)$ on 
${}_{\Q}\overline{X}^{\tau}_\infty$ satisfies a short exact sequence
\begin{equation*}
1 \to  \P_{\tau}(\Q) \to \mathbf H(\Q)_z \to  \mathbf H_{\P,\tau}(\Q)_z \to 1,
\end{equation*}
where $\mathbf H_{\P,\tau}(\Q)_z$ is the stabilizer of $z$ under the action
of $\mathbf H_{\P,\tau}(\Q)$ on $X_{P,\tau}$.
\end{lem}
\begin{proof}
As in the proof of Lemma ~\ref{lemStabilizersRBS}, the normalizer of
$X_{P,\tau}$ is $\P(\Q)$ which acts via its quotient $\P(\Q)/\P_\tau(\Q) =
\mathbf H_{\P,\tau}(\Q)$.
\end{proof}

Similarly to $\overline{X}^{RBS}$, the space
${}_{\Q}\overline{X}_\infty^{\tau}$ is not locally compact and $\ga$ does
not act properly.  Nonetheless one has the
\begin{lem}
\label{lemSatakeDiscontinuous}
The action of an arithmetic subgroup $\ga$ on ${}_{\Q}\overline{X}_\infty^{\tau}$
is discontinuous and the arithmetic quotient
$\ga\backslash {}_{\Q}\overline{X}_\infty^{\tau}$ is a compact Hausdorff space.
\end{lem}

The proof is similar to Lemma ~\ref{lemRBSDiscontinuous} since the fibers
of $p'$ are again compact, being reductive Borel-Serre compactifications of
the $W_{P^\dag,\tau}$.  The \emph{Satake compactification} of
$\ga\backslash X_\infty$ associated to $\tau$ is $\ga\backslash
{}_{\Q}\overline{X}_\infty^{\tau}$.

In the case when the representation $\tau$ is generic one obtains the
maximal Satake compactification $\overline{X}_\infty^{\max}$.  This is
always geometrically rational and the associated
${}_\Q\overline{X}_\infty^{\max}$ is very similar to
$\overline{X}_\infty^{RBS}$.  Indeed in this case $X_P = X_{P,\tau} \times
({}_\R A_{P}/A_{P})$, where ${}_\R A_{P}$ is defined like $A_P$ but using a
maximal $\R$-split torus instead of a maximal $\Q$-split torus, and the
quotient map simply collapses the Euclidean factor ${}_\R A_{P}/A_{P}$ to a
point.  In particular, if $\Q\text{-rank }\mathbf H = \R\text{-rank }
\mathbf H$, then $\ga\backslash {}_\Q\overline{X}_\infty^{\max} \cong
\ga\backslash\overline{X}_\infty^{RBS}$.

\section{The Bruhat-Tits buildings}
\label{sectBruhatTitsBuildings}

For a finite place $v$, let $k_{v}$ be the completion of $k$ with respect
to a norm associated with $v$. Bruhat and Tits \citelist{\cite{BruhatTits1}
  \cite{BruhatTits2}} constructed a building $X_{v}$ which reflects the
structure of $\G(k_{v})$.  The building $X_{v}$ is made up of subcomplexes
called \emph{apartments} corresponding to the maximal $k_{v}$-split tori in
$\G$ and which are glued together by the action of $\G(k_{v})$.  We give an
outline of the construction here together with the properties of $X_{v}$
which are needed in the sections below; in addition to the original papers,
we benefited greatly from
\citelist{\cite{ji}*{\S3.2}\cite{Landvogt}\cite{Tits}}.

In this section we fix a finite place $v$ and a corresponding discrete
valuation $\omega$.

\subsection{The apartment}

Let $\Split$ be a maximal $k_{v}$-split torus in $\G$ and let
$X^{*}(\Split)= \Hom_{k_{v}}(\Split, \G_{m})$ and $X_{*}(\Split)
=\Hom_{k_{v}}(\G_{m}, \Split)$ denote the $k_v$-rational characters and
cocharacters of $\Split$ respectively.  Denote by $\Phi \subset
X^{*}(\Split)$ the set of $k_{v}$-roots of $\G$ with respect to
$\Split$. Let $\N$ and $\Cent$ denote the normalizer and the centralizer,
respectively, of $\Split$; set $N=\N(k_{v})$, $Z=\Cent(k_{v})$. The  Weyl
group $W =
N/Z$ of $\Phi$ acts on the real vector space
\begin{equation*}
V = X_{*}(\Split) \otimes_{\Z}\R = \Hom_{\Z}(X^{*}(\Split) , \R)
\end{equation*}
by linear transformations; for $\alpha\in\Phi$, let $r_\alpha$ denote the
corresponding reflection of $V$.

Let $A$ be the affine space underlying $V$ and let $\Aff(A)$ denote the
group of invertible affine transformations.  We identify $V$ with the
translation subgroup of $\Aff(A)$.  There is an action of $Z$ on $A$ via
translations, $\nu\colon Z\rightarrow V \subset \Aff(A)$, determined by
\begin{equation*}
\chi(\nu(t)) = -\omega(\chi(t))\ , \quad t\in Z,\ \chi\in X^{*}(\Cent)\ ;
\end{equation*}
note that $V = \Hom_{\Z}(X^{*}(\Cent), \R)$ since
$X^{*}(\Cent) \subseteq X^{*}(\Split)$ is a finite index subgroup. 

We now extend $\nu$ to an action of $N$ by affine transformations.  Let $H
= \ker\nu$, which is the maximal compact subgroup of $Z$. Then $Z/H$ is a
free abelian group with rank $= \dim_\R V = k_{v}\rank \G$.  The group $W'
= N/H$ is an extension of $W$ by $Z/H$ and there exists an affine action of
$W'$ on $A$ which makes the following diagram commute
\cite{Landvogt}*{1.6}:
\begin{equation*}
\begin{CD}
1 @>>> Z/H @>>> W' @>>> W @>>> 1 \\
@. @VVV @VVV @VVV \\
1 @>>> V @>>> \Aff(A) @>>> \mathrm{GL}(V) @>>> 1\rlap{\ .}
\end{CD}
\end{equation*}
The action of $W'$ lifts to the desired extension $\nu\colon N \to \Aff(A)$.

For each $\alpha \in \Phi$, let $U_{\alpha}$ be the $k_v$-rational points
of the connected unipotent subgroup of $\G$ which has Lie algebra spanned
by the root spaces $\mathfrak g_\alpha$ and (if $2\alpha$ is a root)
$\mathfrak g_{2\alpha}$.  For $u\in U_{\alpha}\setminus \{1\}$, let $m(u)$
be the unique element of $N\cap U_{-\alpha}uU_{-\alpha}$
\cite{Landvogt}*{0.19}; in $\SL_2$, for example,
$m\left(\left(\begin{smallmatrix} 1 & x \\ 0\vphantom{x^{-1}} &
  1 \end{smallmatrix}\right)\right) = \left(\begin{smallmatrix} 0 & x
  \\ -x^{-1} & 0 \end{smallmatrix}\right)$.  The element $m(u) \in N$ acts
on $A$ by an affine reflection $\nu(m(u))$ whose associated linear
transformation is $r_\alpha$.  The hyperplanes fixed by these affine
reflections for all $\alpha$ and $u$ are the \emph{walls} of $A$.  The
connected components of the complement of the union of the walls are called
the \emph{chambers} of $A$; since we assume $\G$ is almost simple, these
are (open) simplices.  A \emph{face} of $A$ is an open face of a
chamber.  The affine space $A$ is thus a simplicial complex (with the open
simplices being faces) and the action of $N$ is simplicial.

For convenience we identify $A$ with $V$ by choosing a ``zero'' point $o\in
A$.  For $\alpha \in \Phi$, define $\phi_\alpha\colon U_{\alpha} \to \R
\cup \{\infty\}$ by setting $\phi_\alpha(1)=\infty$ and requiring for
$u\neq 1$ that the function $x\mapsto \alpha(x) + \phi_\alpha(u)$ vanishes
on the wall fixed by $\nu(m(u))$.  For $\ell \in \R$, let
\begin{equation*}
U_{\alpha,\ell} = \{\, u\in U_{\alpha} \mid \phi_\alpha(u) \ge \ell\,\}\ .
\end{equation*}
These are compact open subgroups and define a decreasing exhaustive and
separated filtration of $U_{\alpha}$ which has ``jumps'' only for $\ell$ in
the discrete set $\phi_\alpha( U_{\alpha}\setminus \{1\})$.  The affine
function $\alpha + \ell$ is called an \emph{affine root} if for some $u\in
U_{\alpha}\setminus \{1\}$, $\ell = \phi_\alpha(u)$ and (if $2\alpha$ is a
root) $\phi_\alpha(u)= \sup \phi_\alpha(u U_{2\alpha})$; let
$r_{\alpha,\ell} = \nu(m(u))$ be the corresponding affine reflection.  Note
that the zero set of an affine root is a  wall of
$A$ and every wall of $A$ arises in this fashion.

Denote the set of affine roots by $\Phi\aff$; it is an \emph{affine root
  system} in the sense of \cite{Macdonald}.  The Weyl group $W\aff$ of the
affine root system $\Phi\aff$ is the group generated by $r_{\alpha,\ell}$
for $\alpha + \ell \in \Phi\aff$; it is an affine Weyl group in the sense
of \cite{Bourbaki}*{Ch.~VI, \S2} associated to a reduced root system (not
necessarily $\Phi$).  Since we assume $\G$ is simply connected, $W\aff =
\nu(N) \cong W'$.

The \emph{apartment} associated to $\Split$ consists of the affine
simplicial space $A$ together with the action of $N$, the affine root
system $\Phi\aff$, and the filtration of the root groups,
$(U_{\alpha,\ell})_{\substack{\alpha\in\Phi \\ \ell\in \R}}$.

\subsection{The building}
\label{ssectBuilding}

For $x\in A$, let $U_x$ be the group generated by $U_{\alpha,\ell}$ for all
$\alpha + \ell \in\Phi\aff$ such that $(\alpha + \ell)(x) \ge 0$.  The
\emph{building} of $\G$ over $k_v$ is defined \cite{BruhatTits1}*{(7.4.2)}
to be
\begin{equation*}
X_v =  (G\times A ) / \!\sim \ ,
\end{equation*}
where $(gnp,x) \sim (g, \nu(n)x)$ for all $n\in N$ and $p \in H U_x$.  We
identify $A$ with the subset of $X_v$ induced by $\{1\}\times A $.

The building $X_v$ has an action of $\G(k_v)$ induced by left
multiplication on the first factor of $G\times A$.  Under this action, $N$
acts on $A\subset X_v$ via $\nu$ and $U_{\alpha,\ell}$ fixes the points in
the half-space of $A$ defined by $\alpha +\ell\ge 0$.  The simplicial
structure on $A$ induces one on $X_v$ and the action of $\G(k_v)$ is
simplicial.  The subcomplex $gA\subset X_v$ may be identified canonically with the
apartment corresponding to the maximal split torus $g\Split g^{-1}$.

Choose an inner product on $V$ which is invariant under the Weyl group $W$;
the resulting metric on $A$ may be transferred to any apartment by using the
action of $\G(k_v)$.  These metrics fit together to give a well-defined
metric on $X_{v}$ which is invariant under $\G(k_{v})$
\cite{BruhatTits1}*{(7.4.20)} and complete \cite{BruhatTits1}*{(2.5.12)}.
Given two points $x$, $y\in X_v$, there exists an apartment $gA$ of $X_v$
containing them \cite{BruhatTits1}*{(7.4.18)}.  Since $gA$ is an affine
space we can connect $x$ and $y$ with a line segment, $t \mapsto tx +
(1-t)y$, $ t \in [0,1]$; this segment is independent of the choice of
apartment containing the two points and in fact is the unique geodesic
joining $x$ and $y$.

\begin{prop}[\cite{BruhatTits1}*{(7.4.20)}]
The mapping $t \mapsto tx + (1-t)y$ of $[0,1] \times X_{v} \times X_{v}
\rightarrow X_{v}$ is continuous and thus $X_{v}$ is contractible.
\end{prop}

In fact it follows from \cite{BruhatTits1}*{(3.2.1)} that $X_v$ is a
$\CAT(0)$-space\,\footnote{Recall that a $\CAT(0)$-space is a metric space where
the distance between any two points is realized by a geodesic and every
geodesic triangle is thinner than the corresponding triangle of the same
side lengths in the Euclidean plane; see \cite{BH} for a comprehensive
discussion of $\CAT(0)$-spaces. Besides affine buildings such as $X_v$,
simply connected, non-positively curved Riemannian manifolds such as $X_\infty$ are $\CAT(0)$-spaces.}.

\subsection{Stabilizers}
\label{ssectStabilizersBuilding}

For $\Omega \subset X_{v}$, let $\G(k_{v})_{\Omega}$ be the fixing subgroup of $\Omega$) (see 2.1).  Suppose now
that $\Omega \subseteq A$ and set \begin{equation*}
U_{\Omega} = \langle \, U_{\alpha,\ell} \mid (\alpha+\ell)(\Omega)  \geq 0,\,
\alpha+\ell\in \Phi\aff\, \rangle\ .
\end{equation*}
Since $\G$ is simply connected and the valuation $\omega$ is discrete,
$\G(k_{v})_{\Omega} = HU_{\Omega}$ (see \cite{BruhatTits1}*{(7.1.10),
  (7.4.4)}).  In particular, the stabilizer of $x\in A$ is the compact
open subgroup $\G(k_{v})_x = HU_x$.

If $F$ is a face of $A$ and $x\in F$, then the set of affine roots which
are nonnegative at $x$ is independent of the choice of $x\in F$.  Thus
$\G(k_{v})_{F} = \G(k_{v})_x$.  Note that an element of $\G(k_v)$ which
stabilizes $F$ also fixes the barycenter $x_F$ of $F$; thus $\G(k_v)_{x_F}$ is
the stabilizer subgroup $\G(k_v)^F$of $F$.  
\begin{rem}
The stabilizer subgroups for the building
of $\SL_2$ (a tree) are calculated in \cite{SerreTrees}*{II, 1.3}.
\end{rem}

Let $\P$ be a parabolic $k_v$-subgroup which without loss of generality we
may assume contains the centralizer of $\Split$; let $\n_{\P}$ be its
unipotent radical.  Let $\Phi_P = \{\, \alpha \in \Phi \mid U_\alpha
\subseteq \n_{\P}(k_v) \, \}$ and set $E_P = \{\, v\in V \mid \alpha(v) \ge
0, \, \alpha \in \Phi_P\, \}$; note that $\Phi_P$ is contained in a positive
system of roots and hence $E_P$ is a cone with nonempty interior.

\begin{lem}
\label{lemUnipotentsHaveFixedPoints}
For  each $u \in \n_{\P}(k_v)$ there exists $x\in A$ such that  $x + E_P$ is
fixed pointwise by $u$.  In particular, $u$ belongs to a compact open subgroup.
\end{lem}

\begin{proof}
Since $\n_{\P}(k_v)$ is generated by $(U_\alpha)_{\alpha\in\Phi_P}$, there
exists $\ell\in \R$ such that $u$ belongs to the group generated by
$(U_{\alpha,\ell})_{\alpha\in\Phi_P}$.  Since $U_{\alpha,\ell}$ fixes the
points in the half-space of $A$ defined by $\alpha +\ell\ge 0$,
choosing $x\in A$ such that $\alpha(x) \ge -\ell$ for all $\alpha\in
\Phi_P$ suffices.
\end{proof}

\section{The reductive Borel-Serre and Satake compactifications: the
  $S$-arithmetic case}
\label{sectCompactificationsSArithmetic}

We now consider a general $S$-arithmetic subgroup $\ga$ and define a
contractible space $X=X_S$ on which $\ga$ acts properly.  If the $k$-rank
of $\G$ is positive, as we shall assume, $\ga\backslash X$ is noncompact
and it is important to compactify it.  Borel and Serre \cite{BS2} construct
$\olsp^{BS}$, the analogue of $\ga\backslash\overline{X}_\infty^{BS}$ from
\S\ref{ssectBSarith}, and use it to study the cohomological finiteness of
$S$-arithmetic subgroups.  In this section we recall their construction and
define several new compactifications of $\ga\backslash X$ analogous to
those in \S\ref{sectCompactificationsArithmetic}.

\subsection{\boldmath The space $\ga\backslash X$ associated to an
$S$-arithmetic group}

Let $S$ be a finite set of places of $k$ containing the infinite places
$S_\infty$ and let $S_f = S \setminus S_\infty$.  Define
\begin{equation*}
G =G_{\infty}\times \prod_{v\in S_{f}} \G(k_{v}),
\end{equation*}
which is a locally
compact group, and
\begin{equation*}
X =X_{\infty}\times \prod_{v\in S_{f}} X_{v}\ ,
\end{equation*}
where $X_v$ is the Bruhat-Tits building associated to $\G(k_v)$ as
described in \S\ref{sectBruhatTitsBuildings}.  If we need to make clear the
dependence on $S$, we write $X_S$.  $X$ is a locally compact
metric space under the distance function induced from the factors.  Since
each factor is contractible (see \S\ref{ssectBuilding}), the same is true for $X$.

The group $G$ acts isometrically on $X$.  We view $\G(k)\subset G$ under
the diagonal embedding.  Any $S$-arithmetic subgroup $\ga \subset \G(k)$ is
a discrete subgroup of $G$ and acts properly on $X$ \cite{BS2}*{(6.8)}.
It is known that the quotient $\ga\backslash X$ is compact if and
only if the $k$-rank of $\G$ is equal to 0.  In the following, we assume
that the $k$-rank of $\G$ over $k$ is positive.  Then for every $v\in
S_{f}$, the $k_{v}$-rank of $\G$ is also positive.

\subsection{The Borel-Serre compactification \cite{BS2}}
\label{ssectBorelSerreSarithmetic}
Define
\begin{equation*}
\oX^{BS} = \overline{X}_\infty^{BS}\times \prod_{v\in S_{f}}X_{v} \ ,
\end{equation*}
where $\overline{X}_\infty^{BS}$ is as in \S\ref{ssectBSarith}.  This space
is contractible and the action of $\G(k)$ on $X$ extends to a continuous
action on $\overline X^{BS}$.  The action of any $S$-arithmetic subgroup
$\ga$ on $\overline{X}^{BS}$ is proper \cite{BS2}*{(6.10)}.  When
$S_f=\emptyset$ this is proved in \cite{Borel-Serre} as mentioned in
\S\ref{ssectBSarith}; in general, the argument is by induction on $|S_f|$.
Using the barycentric subdivision of $X_v$, the key points are \cite{BS2}*{(6.8)}: 
\begin{enumerate}
\item The covering of $X_v$ by open stars $V(F)$ about the barycenters
  of faces $F$ satisfies
\begin{equation*}
\gamma V(F)\cap V(F) \neq \emptyset \quad \Longleftrightarrow \quad
\gamma\in\Gamma_{F} = \ga\cap \G(k_v)_F \text{ , and}
\end{equation*}
\item For any face $F \subset X_{v}$,  $\ga_F$ is an
  $(S\setminus\{v\})$-arithmetic subgroup and hence by induction acts
  properly on $\overline{X}_{S\setminus\{v\}}^{BS}$.
\end{enumerate}
Furthermore $\ga\backslash \oX^{BS}$ is compact Hausdorff
\cite{BS2}*{(6.10)} which follows inductively from
\begin{enumerate}[resume]
\item There are only finitely many $\ga$-orbits of simplices in $X_{v}$ for
  $v\in S_f$ and the quotient of  $\overline{X}^{BS}_\infty$ by an
  arithmetic subgroup is compact.
\end{enumerate}

\subsection{The reductive Borel-Serre compactification}
\label{subsectRBSSarith}
Define
\begin{equation*}
\oX^{RBS}= \overline{X}_\infty^{RBS}\times \prod_{v\in S_f}X_{v} \ .
\end{equation*}
There is a $\G(k)$-equivariant surjection $\overline{X}^{BS} \to
\overline{X}^{RBS}$ induced from the surjection  in \S\ref{subsectRBSarith}.
\begin{prop}
\label{propDiscontinuousRBS}
Any $S$-arithmetic subgroup $\ga$ of $\G(k)$ acts discontinuously on
$\oX^{RBS}$ with a compact Hausdorff quotient $\ga\backslash\oX^{RBS}$.
\end{prop}
The proposition is proved similarly to the case of $\olsp^{BS}$ outlined in
\S\ref{ssectBorelSerreSarithmetic}; one replaces ``proper'' by
``discontinuous'' and begins the induction with Lemma
~\ref{lemRBSDiscontinuous}.  The space $\ga\backslash\oX^{RBS}$ is the
\emph{reductive Borel-Serre compactification} of $\ga\backslash X$.

\subsection{Satake compactifications}
\label{subsectSatakeSArith}
Let $(\tau,V)$ be a spherical representation of
$\operatorname{Res}_{k/\Q}\G$ as in \S\ref{subsectSatakeArith} and define
\begin{equation*}
{}_{\Q}\overline{X}^{\tau}=
{}_{\Q}\overline{X}_\infty^{\tau}
\times\prod_{v\in S_{f}} X_{v}\ .
\end{equation*}
There is a $\G(k)$-equivariant surjection $\overline{X}^{RBS} \to
{}_{\Q}\overline{X}^{\tau}$ induced by  $\overline{X}_\infty^{RBS} \to
{}_{\Q}\overline{X}_\infty^{\tau}$ from \S\ref{subsectSatakeArith}.

\begin{prop}
\label{propDiscontinuousSatake}
Assume that the Satake compactification $\overline{X}_\infty^{\tau}$ is
geometrically rational.  Then any $S$-arithmetic subgroup $\ga$ acts
discontinuously on ${}_{\Q}\overline{X}^{\tau}$ with a compact Hausdorff
quotient $\ga\backslash {}_{\Q}\overline{X}^{\tau}$.
\end{prop}

The compact quotient $\ga\backslash {}_{\Q}\overline{X}^{\tau}$ is
called the \emph{Satake compactification} associated with $(\tau,V)$.

\section{The fundamental group of the compactifications}
\label{sectFundGrpArithmetic}

In this section we state our main result, Theorem ~\ref{thmMainArithmetic},
which calculates the fundamental group of the reductive Borel-Serre and the
Satake compactifications of $\ga\backslash X$. 
The proof of Theorem ~\ref{thmMainArithmetic} is postponed to
\S\ref{sectProofArithmetic}.

Throughout we fix a spherical representation $(\tau,V)$ such that
$\overline{X}_\infty^{\tau}$ is geometrically rational.

In our situation of an $S$-arithmetic subgroup $\ga$ acting on $\oX^{RBS}$
and ${}_{\Q}\overline{X}^{\tau}$, we denote $\ga_{fix}$ by $\ga_{fix,RBS}$ and
$\ga_{fix,\tau}$ respectively (see section 2.1 for the definition of $\ga_{fix}$).  The main result of this paper is the
following theorem.

\begin{thm}
\label{thmMainArithmetic}
For any $S$-arithmetic subgroup $\ga$, there exists a commutative diagram
\begin{equation*}
\begin{CD}
\pi_{1}(\olX^{RBS}) @<\cong<<  \ga/\ga_{fix,RBS} \\
@VVV                       @VVV \\
\pi_{1}(\ga\backslash {}_{\Q}\overline{X}^{\tau}) @<\cong<<
\ga/\ga_{fix,\tau}
\end{CD}
\end{equation*}
where the horizontal maps are isomorphisms and the vertical maps are
surjections induced by the $\ga$-equivariant projection $\oX^{RBS} \to
{}_{\Q}\overline{X}^{\tau}$ and the inclusion $\ga_{fix,RBS} \subseteq
\ga_{fix,\tau}$.
\end{thm}

The proof of the theorem will be given in \S\ref{sectProofArithmetic}.  In
the remainder of this section we will deduce from it more explicit
computations of the fundamental groups. To do this we first need to calculate
$\ga_{fix,RBS}$ and $\ga_{fix,\tau}$ which will require the information on
stabilizers from \S\S\ref{subsectRBSarith}, \ref{subsectSatakeArith}, and
\ref{ssectStabilizersBuilding}.

Let $\P$ be a parabolic $k$-subgroup $\P$ of $\G$.  The $S$-arithmetic
subgroup $\ga$ induces $S$-arithmetic subgroups $\ga_{P}=\ga\cap
\P(k)\subseteq \P(k)$, $\ga_{N_{P}} = \ga\cap \n_{\P}(k) \subseteq
\n_{\P}(k)$, and $\ga_{L_{P}} = \ga_{P}/\ga_{N_{P}} \subseteq \mathbf
L_{\P}(k)$, as well as $\ga_{P_\tau} = \ga\cap \P_\tau(k) \subseteq
\P_\tau(k)$ and $\ga_{H_{P,\tau}} = \ga_P / \ga_{P_\tau} \subseteq \mathbf
H_{\P, \tau}(k)$.

Let $E\Gamma\subseteq \ga$ be the subgroup generated by $\Gamma_{N_{P}}$
as $\P$ ranges over all parabolic $k$-subgroups  of $\G$. Since $\gamma
\n_{\P}\gamma^{-1}=\n_{\gamma \P\gamma^{-1}}$ for $\gamma \in \Gamma$,
$E\Gamma$ is clearly normal. Let $E_{\tau}\Gamma\subseteq \ga$ be the
subgroup generated by $\Gamma_{P_\tau} \cap \bigcap_{v\in S_f} K_v$ for
every $\tau$-saturated parabolic $k$-subgroup $\P$ of $\G$ and compact open
subgroups $K_v\subset \G(k_v)$.  As above, $E_{\tau}\Gamma$ is normal.
Since $\ga_{N_P}$ is generated by $\ga_{N_P} \cap \bigcap_{v\in S_f} K_v$
for various $K_v$ by Lemma~\ref{lemUnipotentsHaveFixedPoints}, it is easy
to see that $E\ga \subseteq E_\tau\ga$.

A subgroup $\ga\subset \mathbf G(k)$ is \emph{neat} if the subgroup of $\C$
generated by the eigenvalues of $\rho(\gamma)$ is torsion-free for any
$\gamma\in\ga$.  Here $\rho$ is a faithful representation $\G\to \GL_N$
defined over $k$ and the condition is independent of the choice of $\rho$.
Clearly any neat subgroup is torsion-free.  Any $S$-arithmetic subgroup has a
normal neat subgroup of finite index \cite{Borel}*{\S17.6}; the image of a
neat subgroup by a morphism of algebraic groups is neat
\cite{Borel}*{\S17.3}.

\begin{prop}
\label{propGammaFixedIsEGamma}
Let $\Gamma$ be an $S$-arithmetic subgroup. Then $E\Gamma \subseteq
\Gamma_{fix,RBS}$ and $E_{\tau}\Gamma \subseteq \Gamma_{fix,\tau}$.  If $\Gamma$
is neat then equality holds for both.
\end{prop}

\begin{proof}
We proceed by induction on $\vert S_{f}\vert$. Suppose first that
$\vert S_{f}\vert=0$.  By Lemma ~\ref{lemStabilizersRBS}, $\Gamma_{N_P}$
stabilizes every point of $X_{P} \subseteq \overline{X}^{RBS}_\infty$ for any
parabolic $k$-subgroup $\P$, and hence $E\Gamma \subseteq \Gamma_{fix,RBS}$.
Likewise by Lemma ~\ref{lemStabilizersSatake}, $\ga_{P_\tau}$ stabilizes
every point of $X_{P,\tau} \subset {}_{\Q}\overline{X}^{\tau}_\infty$ and so
$E_{\tau}\Gamma \subseteq \Gamma_{fix,\tau}$.

If $\Gamma$ is neat, then $\Gamma_{L_{P}}$ and $\Gamma_{H_{P,\tau}}$ are
likewise neat and hence torsion-free.  The actions of $\ga_{L_P}$ and
$\ga_{H_{P,\tau}}$ are proper and hence $\Gamma_{L_{P},z}$ and
$\Gamma_{H_{P,\tau},z}$ are finite.  Thus these stabilizer subgroups must
be trivial.  It follows then from Lemmas ~\ref{lemStabilizersRBS} and
\ref{lemStabilizersSatake} that $E\Gamma = \Gamma_{fix,RBS}$ and
$E_{\tau}\Gamma = \Gamma_{fix,\tau}$.

Now suppose that $\vert S_{f}\vert > 0$, pick $v \in S_{f}$ and let $S' = S \setminus\{v\}$.  Write
$\overline{X}^{RBS} = \overline{X}_{S'}^{RBS} \times X_v$.  Suppose that
$\gamma\in \ga_{N_P}$ for some parabolic $k$-subgroup $\P$.  By Lemma
~\ref{lemUnipotentsHaveFixedPoints}, $\gamma \in \G(k_v)_y$ for some $y\in
X_{v}$.  Thus $\gamma \in \ga' \cap \n_{\P}(k)$, where $\ga' = \ga\cap
\G(k_v)_y$.  Since $\G(k_v)_y$ is a compact open subgroup, $\ga'$ is an
$S'$-arithmetic subgroup.  Since $\vert S'_{f}\vert < \vert S_{f}\vert$, we can apply
induction to see that $\gamma = \gamma_1 \dots \gamma_m$,
where $\gamma_i\in\ga'_{x_i}$ with $x_i\in \overline{X}_{S'}^{RBS}$.  Since
each $\gamma_i\in \Gamma_{(x_i,y)} \subset \Gamma_{fix,RBS}$, we see $E\Gamma
\subseteq \Gamma_{fix,RBS}$.  The proof that $E_{\tau}\Gamma \subseteq
\Gamma_{fix,\tau}$ is similar since if $\gamma \in \Gamma_{P_\tau} \cap
\bigcap_{v\in S_f} K_v$ then $\gamma \in \G(k_v)_y$ for some $y\in X_v$
\cite{BruhatTits1}*{(3.2.4)}.

Assume that $\Gamma$ is neat.  Let $(x,y) \in \overline{X}_{S'}^{RBS}\times
X_{v}$, and let  $F$ be the face of $X_v$ containing $y$.  As above,
$ \ga_F = \ga\cap \G(k_v)_F$ is $S'$-arithmetic and, in this case, neat.
So by induction, $\Gamma_{F,x} \subseteq E(\ga_F)\subseteq E\Gamma$.  But
since $\G(k_{v})_{y}=\G(k_{v})_{F}$, $\Gamma_{(x,y)} =
\Gamma_{F,x}$. Therefore $\Gamma_{fix,RBS} \subseteq E\Gamma $. A similar
argument shows that $\Gamma_{fix,\tau} \subseteq E_{\tau}\Gamma $.
\end{proof}

We now can deduce several corollaries of Theorem ~\ref{thmMainArithmetic}
and Proposition ~\ref{propGammaFixedIsEGamma}.

\begin{cor}
\label{corNeat}
$\pi_{1}(\olX^{RBS})$ is a quotient of $\ga/ E\ga$ and
$\pi_{1}(\ga\backslash {}_{\Q}\overline{X}^{\tau})$ is a quotient of
$\ga/ E_\tau\ga$.  If $\ga$ is  neat, then $\pi_{1}(\olX^{RBS}) \cong
\ga/ E\ga$ and $\pi_{1}(\ga\backslash
{}_{\Q}\overline{X}^{\tau}) \cong \ga/ E_\tau\ga$.
\end{cor}

\begin{cor}
If $k\rank \G >0 $ and $S\rank \G \ge 2$, $\pi_{1}(\olX^{RBS})$ and
$\pi_{1}(\ga\backslash {}_{\Q}\overline{X}^{\tau})$  are finite.
\end{cor}
\begin{proof}
Under the rank assumptions, $E\ga$ is $S$-arithmetic
\citelist{\cite{Margulis} \cite{Ra2}*{Theorem~ A, Corollary~ 1}}.
\end{proof}

\section{Proof of the main theorem}
\label{sectProofArithmetic}
In this section we give the proof of Theorem ~\ref{thmMainArithmetic}. If the two
spaces $\overline{X}^{RBS}$ and ${}_{\Q}\overline{X}^{\tau}$ were simply connected,
and if the actions of $\Gamma$ on them were free, then the path-lifting property of
covering spaces would make computing the fundamental groups of the quotient spaces
elementary. However, the actions of $\Gamma$ are not free; in fact Lemmas  ~\ref{lemStabilizersRBS}
and \ref{lemStabilizersSatake} show that
the  $\Gamma$-stabilizers of points in these spaces may not be finite. Nonetheless, if  $\Gamma$ is neat
the quotient maps do satisfy a weaker path-lifting property, \emph{admissibility} (Proposition ~\ref{propAdmissibleNeatCase}).
This was introduced by Grosche in \cite{Gro}. Provided that the spaces are simply connected,
this property makes it possible to compute their fundamental  groups (\cite{Gro}*{Satz~5}). 

To show that $\overline{X}^{RBS}$ and ${}_{\Q}\overline{X}^{\tau}$ are simply connected, we
prove that the natural surjections of $\overline{X}^{BS}$ onto each of them also have this path-lifting property
(Proposition ~\ref{propSimply}). Since $\overline{X}^{BS}$ is contractible it follows then that the two spaces are simply connected (Lemma~\ref{propAdmissibilityImpliesSC}).

To prove that the quotient maps are admissible, we use the fact that admissibility is a local property (Lemma~\ref{lemAdmissibilityIsLocal}). So we can cut a path into finitely many pieces, each of which lies in a ``nice'' neighborhood. We then push each such segment out to the boundary where the geometry
is simpler. This technique is formalized in Lemma ~\ref{lemAdmissibilityViaRetract}. The construction of suitable neighborhoods is via reduction theory.

It is easy to visualize this in the case where
$\R\rank \G =1$ and $\Gamma$ is arithmetic. Then $\overline{X}^{RBS} = X \cup \{\rm cusps\}$.
Each cusp has horospherical neighborhoods which retract onto it. Suppose we have a path
$\omega$ in $\overline{X}^{RBS}$. Let $y \in \overline{X}^{RBS}$ be a cusp and
$U \subset \overline{X}^{RBS}$ be a horospherical neighborhood  of $y$ such that
$\omega|_{[t_0,t_1]}$ lies in $U$ with $x_0=\omega(t_0), x_1=\omega(t_1) \in X\cap U$.
Let $r_t$ be the deformation retraction of $U$ onto $y$ and set $\sigma_i(t)=r_t(x_i)$, $i=0,1$. Then $\omega|_{[t_0,t_1]} 
\cong\sigma_0^{-1} \cdot \sigma_1 \ \rm rel\{x_0, x_1\}$. Deforming paths in this way
makes it easy to see that $\overline{X}^{RBS}$ is simply connected and to compute the fundamental
group of $\Gamma\backslash\overline{X}^{RBS}$.

Throughout homotopy of paths $\omega$ and $\eta$ will always mean homotopy relative to
the endpoints and will be denoted $\omega \cong \eta$.  An action of a
topological group $\ga$ on a topological space $Y$ will always be a
continuous action.

\begin{defi}
\label{defiAdmissible}(cf. \cite{Gro}*{Definition~3})
A continuous surjection $p\colon Y \to X$ of topological spaces is
\emph{admissible} if for any path $\omega$ in $X$ with initial point $x_0$
and final point $x_1$ and for any $y_0\in p^{-1}(x_0)$, there exists a path
$\tilde{\omega}$ in $Y$ starting at $y_0$ and ending at some $y_1\in p^{-1}(x_1)$
such that $p\circ \tilde \omega$ is homotopic to
$\omega$ relative to the endpoints. An action of a group $\ga$
on a topological space $Y$ is \emph{admissible} if the quotient map $Y\to
\ga\backslash Y$ is admissible.
\end{defi}

\begin{prop}
\label{propGrosche}
Let $Y$ be a simply connected topological space and $\ga$ a discrete group
acting on $Y$.  Assume that either
\begin{enumerate}
\item\label{itemGrosche} the $\ga$-action is discontinuous and admissible,
  or that
\item\label{itemArmstrong} the $\ga$-action is proper and $Y$ is a locally
  compact metric space.
\end{enumerate}
Then the natural morphism $\ga \to \pi_{1}(\ga\backslash Y)$ induces an
isomorphism $\ga/\ga_{fix} \cong \pi_{1}(\ga\backslash Y)$.
\end{prop}
\begin{proof}
See \cite{Gro}*{Satz~5} and \cite{Armstrong} for hypotheses
\ref{itemGrosche} and \ref{itemArmstrong} respectively .
\end{proof}

\begin{prop}
\label{propAdmissibilityImpliesSC}
Let $p\colon Y \to X$ be an admissible continuous map of a simply
connected topological space $Y$ and assume that $p^{-1}(x_0)$ is
path-connected for some $x_0\in X$.  Then $X$ is simply connected.
\end{prop}
\begin{proof}
Let $\omega\colon [0,1] \to X$ be a loop based at $x_0$ and let
$\tilde\omega$ be a path in $Y$ such that $p\circ \tilde\omega \cong
\omega$ (relative to the basepoint).  Let $\eta$ be a path in
$p^{-1}(x_0)$ from $\tilde\omega(1)$ to $\tilde\omega(0)$.  Then the
product $\tilde\omega\cdot \eta$ is a loop in the simply connected space
$Y$ and hence is null-homotopic.  It follows that $\omega\cong p\circ
\tilde \omega\cong p\circ(\tilde\omega\cdot \eta)$ is null-homotopic.
\end{proof}

\begin{lem}
\label{lemAdmissibilityIsLocal}
A continuous surjection $p\colon Y \to X$ of topological spaces is
admissible if and only if $X$ can be covered by open subsets $U$
such that $p|_{p^{-1}(U)}\colon p^{-1}(U) \to U$ is
admissible.
\end{lem}
\begin{proof}
By the Lebesgue covering lemma, any path $\omega\colon [0,1] \to X$ is
equal to the product of finitely many paths, each of which maps into
one of the subsets $U$.  The lemma easily follows.
\end{proof}

\begin{lem}
\label{lemAdmissibilityViaRetract}
Let $p\colon Y \to X$ be a continuous surjection of topological spaces.
Assume there exist deformation retractions $r_t$ of $X$ onto a subspace
$X_0$ and $\tilde r_t$ of $Y$ onto $Y_0 = p^{-1}(X_0)$ such that $p\circ
\tilde r_t = r_t \circ p$.  Also assume for all $x\in X$ that
$\pi_0(p^{-1}(x)) \xrightarrow{\tilde r_{0*}} \pi_0(p^{-1}(r_0(x)))$ is
surjective.  Then $p$ is admissible if and only if $p|_{Y_0}\colon Y_0\to
X_0$ is admissible.
\end{lem}

\setlength{\pinch}{.002128769252056923\textwidth}
\setlength{\mim}{2.85427559055181102\pinch}
\begin{figure}[h]
\begin{equation*}
\begin{xy}
%
<0\mim,-15\mim>;<3\mim,-15\mim>:
<24\mim,-3\mim>="c"+<0\mim,\g>="cdmid"+<0\mim,\g>="d",
(15,0)="adown";(0,10)="bdown" **[bordergrey]\crv{(10,5)&(5,8)}?(.3)="xdown",
?(.25)="main3",?(.6)="main6",
"adown"+"c";{"bdown"+"c"} **[verylightgrey]\crv{ (10,5)+"c" & (5,8)+"c"},
(15,10)="aup";(0,20)="bup" **[bordergrey]\crv{(10,15)&(5,18)}
?(.5)="xup",
"aup"+"d";{"bup"+"d"} **[bordergrey]\crv{"d"+(10,15) & "d"+(5,18)},
"adown",\blownupslice{bordergrey}{bordergrey},
"bdown",\blownupslice{verylightgrey}{bordergrey},
"main6",\blownupslice{verylightgrey}{verylightgrey},
"bot";p+<0\mim,42\mim>**\dir{}?(.65)*\dir{*}="y1"*+!L{_{y_1}}="f3",
?(.25)*\cir<1\pinch>{}*\frm{*}="f7",?(.85)*\cir<1\pinch>{}*\frm{*}="f1",
?(.75)*+{}="f2",?(.35)*+{}="f6",?(.45)*\cir<1\pinch>{}*\frm{*}="f5",
?(.55)*+{}="f4",?(.05)*+{}="f9",
"top"+<0\mim,\g>*++!DC\txt<20\mim>\tiny{$ p^{-1}(x_1)$ (marked by
  $\scriptscriptstyle \bullet$)\\$\downarrow$},
%
"main6";"upper" **[lightgrey]\dir{-}?(.5)="r0y1"*\dir{*}*+!RD{_{\tilde
    r_0(y_1)}},
?(.25)="eta1"*\dir{*}*+!UR{_{\eta(1)}},
"lower"-<0\mim,\g>*++!UC\txt<20\mim>\tiny{$\uparrow$\\$p^{-1}(r_0(x_1))$},
;"lower"**[lightgrey]\dir{--},
"upper";"upper"+<0\mim,\g>**[lightgrey]\dir{--},
"main3",\blownupslice{verylightgrey}{verylightgrey},
"main3";"bot" **\crv{~*\dir{} "ccp"},?(.8)="x0bot";p+<0\mim,37.5\mim>="x0top"
**\dir{}?(.4)="y0"*\dir{*}*+!L{_{y_0}},
"main3";"upper" **\dir{}?(.35)="r0y0"*\dir{*}*+!UR{_{\tilde r_0(y_0)}},
;"y0" **\dir{},?(.5)+/u2.25\pinch/="mcp",
"r0y0";"y0" **\crv{ "mcp"},?(.5)*+!U{_{\tilde\sigma_0}},*\dir{>},
"r0y0"+<0\pinch,.5\pinch>;"y0"+<0\pinch,.5\pinch> **\crv{ "mcp"+<0\pinch,.5\pinch>},?(.5)*\dir{>},
"r0y0";"eta1" **\dir{},?(.45)+/r3.75\pinch/="cp1",?(.55)+/l2\pinch/="cp2",
"eta1" **\crv{ "cp1" & "cp2" },?(.5)*+!UR{_\eta},*\dir{>},
"r0y0"+<.35\pinch,.35\pinch>;"eta1"+<.35\pinch,.35\pinch> **\crv{ "cp1"+<.35\pinch,.35\pinch> & "cp2"+<.35\pinch,.35\pinch> },?(.5)*\dir{>},
"eta1";"r0y1" **\dir{-},?(.5)*+!R{_{\psi}},*\dir{>},
"eta1"+<.5\pinch,0\pinch>;"r0y1"+<.5\pinch,0\pinch> **\dir{-},?(.5)*\dir{>},
"r0y1";"y1" **\dir{},?(.5)+/d4.5\pinch/="mcp",
"r0y1";"y1" **\crv{ "mcp"},?(.5)*+!U{_{\tilde\sigma_1}},*\dir{>},
"r0y1"+<0\pinch,.5\pinch>;"y1"+<0\pinch,.5\pinch> **\crv{ "mcp"+<0\pinch,.5\pinch>},?(.5)*\dir{>},
"adown"+<-2\mim,-5\mim>*{_{Y_0\quad\qquad\subseteq \quad\qquad Y}},
%
{<79\mim,15\mim> \ar _{p} @{>} <89\mim,15\mim>},
%
<90\mim,0\mim>;<93\mim,0\mim>:
<24\mim,-3\mim>="c"+<0\mim,\g>="cdmid"+<0\mim,\g>="d",
(15,0)="a";(0,10)="b" **[bordergrey]\crv{(10,5)&(5,8)},
?(.25)="main3"*\dir{*}*+!UR{_{r_0(x_0)}},?(.6)="main6"*\dir{*}*+!UR{_{r_0(x_1)}},,
"main3",\slice{verylightgrey}{verylightgrey},
"main3";"bot" **\crv{~*\dir{} "ccp"},?(.8)="x0"*\dir{*}*+!U{_{x_0}},
"x0" **\crv{ "ccp"},?(.5)*+!U{_{\sigma_0}},*\dir{>},
"main3"+<0\pinch,.5\pinch>;"x0"+<0\pinch,.5\pinch> **\crv{ "ccp"+<0\pinch,.5\pinch>},?(.5)*\dir{>},
"main6",\slice{verylightgrey}{verylightgrey},
"main6"+"cdmid"-<1.5\mim,0\mim>="x1"*\dir{*}*+!DR{_{x_1}},  
"a"+"cdmid";{"b"+"cdmid"} **\crv{~*\dir{} "cdmid"+(10,5) & "cdmid"+(5,8)},
"a"+"c";{"b"+"c"} **[verylightgrey]\crv{"c"+(10,5) & "c"+(5,8)},
"main6";"x1" **\dir{},?(.5)+/d3\pinch/="mcp",
"main6";"x1" **\crv{ "mcp"},?(.6)*+!U{_{\sigma_1}},*\dir{>},
"main6"+<0\pinch,.5\pinch>;"x1"+<0\pinch,.5\pinch> **\crv{ "mcp"+<0\pinch,.5\pinch>},?(.6)*\dir{>},
"x0";"x1" **\dir{},?(.45)+/r37.5\pinch/="cp1",?(.55)+/l27.5\pinch/="cp2",
"x1" **\crv{ "cp1" & "cp2" },?(.5)*+!LD{_\omega},*\dir{>},
"x0"+<.35\pinch,.35\pinch>;"x1"+<.35\pinch,.35\pinch> **\crv{ "cp1"+<.35\pinch,.35\pinch> & "cp2"+<.35\pinch,.35\pinch> },?(.5)*\dir{>},
"main3";"main6" **\crv{(10.2,4.4)}?(.6)*\dir{>},   
*+!UR{_{r_0\circ\omega}},
"main3"+<.35\pinch,.35\pinch>;"main6"+<.35\pinch,.35\pinch> **\crv{(10.2,4.4)+<.35\pinch,.35\pinch>}?(.6)*\dir{>},   
"a",\slice{bordergrey}{bordergrey},
"b",\slice{verylightgrey}{bordergrey},
"a"+"d";{"b"+"d"} **[bordergrey]\crv{"d"+(10,5) & "d"+(5,8)},
"a"+<10\mim,-6\mim>*{_{X_0\quad\subseteq \quad X}}
\end{xy}
\end{equation*}
\caption{$p\colon Y \to X$ as in Lemma~\ref{lemAdmissibilityViaRetract}}
\label{figAdmissibility}
\end{figure}

\begin{proof}
(See Figure ~\ref{figAdmissibility}.) Assume $p|_{Y_0}$ is admissible.  If
  $\omega$ is a path in $X$ from $x_0$ to $x_1$, then $\omega \cong
  \sigma_0^{-1} \cdot (r_0 \circ \omega) \cdot \sigma_1$ where $\sigma_i(t)
  = r_t(x_i)$ for $i=0$, $1$.  Pick $y_0\in p^{-1}(x_0)$ and let $\eta(t)$
  be a path in $Y_0$ starting at $\tilde r_0(y_0)$ such that $p\circ \eta
  \cong r_0\circ \omega$.  By assumption there exists $y_1\in p^{-1}(x_1)$
  such that $\tilde r_0(y_1)$ is in the same path-component of
  $p^{-1}(r_0(x_1))$ as $\eta(1)$; let $\psi$ be any path in
  $p^{-1}(r_0(x_1))$ from $\eta(1)$ to $\tilde r_0(y_1)$.  Set
  $\tilde\omega = \tilde\sigma_0^{-1} \cdot \eta \cdot \psi \cdot
  \tilde\sigma_1$, where $\tilde \sigma_i(t) = \tilde r_t(y_i)$.  Then
  $p\circ \tilde\omega \cong \sigma_0^{-1} \cdot (r_0 \circ \omega) \cdot
  \sigma_1$ and thus $p$ is admissible.
\end{proof}

Recall the $\G(k)$-equivariant quotient maps $\overline{X}^{BS}
\xrightarrow{p_1} \overline{X}^{RBS} \xrightarrow{p_2}
            {}_{\Q}\overline{X}^{\tau}$ from \S\S\ref{subsectRBSSarith},
            \ref{subsectSatakeSArith}.

\begin{prop}
\label{propSimply}
The spaces $\overline{X}^{RBS}$ and ${}_{\Q}\overline{X}^{\tau}$ are
simply connected.
\end{prop}
\begin{proof}
For any finite place $v$, the building $X_{v}$ is contractible. So we need
only prove that $\overline{X}^{RBS}_\infty$ and
${}_{\Q}\overline{X}^{\tau}_\infty$ are simply connected (the case that
$S_{f} = \emptyset$).  By Proposition ~\ref{propAdmissibilityImpliesSC},
Lemma ~\ref{lemAdmissibilityIsLocal}, and the fact that
$\overline{X}^{BS}_\infty$ is contractible, it suffices to find a cover of
$\overline{X}^{RBS}_\infty$ by open subsets $U$ over which $p_1$
(resp. $p_2\circ p_1$) is admissible.

Consider first $\overline{X}^{RBS}_\infty$.  The inverse image
$p_1^{-1}(X_Q)$ of a stratum $X_Q \subseteq \overline{X}^{RBS}_\infty$ is
$e(Q) = N_Q\times X_Q \subseteq \overline{X}^{BS}_\infty$.  Set $\tilde U =
\overline A_Q(1) \times N_Q \times X_Q \subseteq \overline{X}^{BS}_\infty$
(compare \eqref{Pcorner}) and  $U=p_1(\tilde U)$,  a neighborhood of $X_Q$; note $p_1^{-1}(U)= \tilde
U$.  Define a deformation retraction of 
$\tilde U$ onto
$e(Q)$ by
\begin{equation*}
\tilde r_t(a,u,z) =
\begin{cases}
(\exp(\frac{1}{t}\log a), u, z) & \text{for $t\in (0,1]$,} \\
(o_Q,                     u, z) & \text{for $t=0$.}
\end{cases}
\end{equation*}
This descends to a deformation retraction $r_t$ of $U$ onto $X_Q$.  Since
$p_1|_{e(Q)}\colon N_Q\times X_Q \to X_Q$ is admissible and $N_Q$ is
path-connected, Lemma ~\ref{lemAdmissibilityViaRetract} shows that
$p_1|_{\tilde U}$ is admissible.

Now consider ${}_{\Q}\overline{X}^{\tau}_\infty$ and a stratum $X_{Q,\tau}$,
where $\mathbf Q$ is $\tau$-saturated.  The inverse image $(p_2\circ
p_1)^{-1}(X_{Q,\tau})$ is $\coprod_{\P^\dag = \mathbf Q} e(P) \subseteq
\overline{X}^{BS}_\infty$; it is an open subset of the closed stratum
$\overline{e(Q)} = \coprod_{\P \subseteq \mathbf Q} e(P)$.  For each $\P$
such that $\P^\dag = \mathbf Q$ (see \eqref{Pdagger}), we can write $e(P) = N_P\times X_P = N_P
\times X_{Q,\tau} \times W_{P,\tau}$ by \eqref{eqnBoundaryDecomposition}.
Thus $(p_2\circ p_1)^{-1}(X_{Q,\tau}) = Z_Q\times X_{Q,\tau}$, where $Z_Q =
\coprod_{\P^\dag = \mathbf Q} ( N_P \times W_{P,\tau})$.  Note that
$N_Q\times W_{Q,\tau}$ is dense in $Z_Q$, so $Z_Q$ is path-connected.

For $X_{Q,\tau}\subset {}_{\Q}\overline{X}^{\tau}_\infty$, the construction
of $\tilde U$ is more subtle than in the case of
$\overline{X}^{RBS}_\infty$.  The theory of tilings \cite{sap1}*{Theorem
  ~8.1} describes a neighborhood in $\overline{X}^{BS}_\infty$ of the
closed stratum $\overline{e(Q)}$ which is piecewise-analytically
diffeomorphic to $\overline A_Q(1)\times \overline{e(Q)}$.  (Note however
that the induced decomposition on the part of this neighborhood in
$X_\infty(Q)$ does \emph{not} in general agree with that of
\eqref{Pcorner} - see  [33, \S8, Remark (1)].\,)  We thus obtain a neighborhood $\tilde U$ of $(p_2\circ
p_1)^{-1}(X_{Q,\tau}) = Z_Q \times X_{Q,\tau}$ in $\overline{X}^{BS}_\infty$ and a
piecewise-analytic diffeomorphism $\tilde U \cong \overline A_Q(1)\times
Z_Q \times X_{Q,\tau}$; let $U = p_2\circ p_1(\tilde U)$ and note 
$(p_2\circ p_1)^{-1}(U) = \tilde U$.  Since $Z_Q$ is
path-connected, we proceed as in the $\overline{X}^{RBS}_\infty$ case.
\end{proof}

\begin{rem}
It is proved in \cite{ji2} that every Satake compactification
$\overline{X}^{\tau}_\infty$ of a symmetric space $X_\infty$ is a topological ball
and hence contractible. Though the partial Satake compactification
${}_{\Q}\overline{X}^{\tau}_\infty$ is contained in
$\overline{X}^{\tau}_\infty$ as a subset, their topologies are different and
this inclusion is not a topological embedding.  Hence, it does not follow
that ${}_{\Q}\overline{X}^{\tau}_\infty$ is contractible or that a path in
${}_{\Q}\overline{X}^{\tau}_\infty$ can be retracted into the interior.
It can be shown, however, that ${}_{\Q}\overline{X}^{\tau}_\infty$ is weakly
contractible, i.e. that all of its homotopy groups are trivial.
\end{rem}

\begin{prop}\label{propAdmissibleNeatCase}
For any neat $S$-arithmetic subgroup $\ga$, the action of $\ga$ on
$\overline{X}^{RBS}$ and on ${}_{\Q}\overline{X}^{\tau}$ is admissible.
\end{prop}

\begin{proof}
Let $Y = \overline{X}^{RBS}$ or ${}_{\Q}\overline{X}^{\tau}$ and let
$p\colon Y \to \ga\backslash Y$ be the quotient map, which in this case is
open.  It suffices to find for any point $x\in Y$ an open neighborhood $U$
such that $p|_U$ is admissible.  For then $p|_{\ga U}$
is admissible and hence, by Lemma ~\ref{lemAdmissibilityIsLocal}, $p$ is
admissible.

We proceed by induction on $\vert S_{f}\vert$ and we suppose first that
$S_{f}=\emptyset$.

Suppose $x$ belongs to the stratum $X_Q$ of $\overline{X}^{RBS}_\infty$.  Since
$\ga$ is neat, $\ga_{L_Q}$ is torsion-free.  Thus we can choose a
relatively compact neighborhood $O_Q$ of $x$ in $X_Q$ so that
$p|_{O_Q}\colon O_Q \to p(O_Q)$ is a homeomorphism. Let $U =
p_1(\overline A_Q(s) \times N_Q \times O_Q) \subseteq \overline{X}^{RBS}_\infty$
where $s>0$; this is a smaller version of the set $U$ constructed in the
proof of Proposition ~\ref{propSimply}.  By reduction theory, we can choose
$s$ sufficiently large so that the identifications induced by $\ga$ on $U$
agree with those induced by $\ga_Q$ \cite{Zu3}*{(1.5)}.  Since $\ga_Q
\subseteq N_Q \widetilde M_Q $, it acts only on the last two factors of
$\overline A_Q \times N_Q \times X_Q$.  Thus the deformation retraction
$r_t$ of $U$ onto $O_Q$ (from the proof of Proposition ~\ref{propSimply})
descends to a deformation retraction of $p(U)$ onto $p(O_Q)=O_Q$.
Now apply Lemma ~\ref{lemAdmissibilityViaRetract} to see that $p|_U$ is
admissible.

For $x$ in the stratum $X_{Q,\tau}$ of ${}_{\Q}\overline{X}^{\tau}_\infty$, we
again emulate the construction of $U$ from the proof of Proposition
~\ref{propSimply}.  Specifically let $U= (p_2\circ p_1)(\overline
A_Q(s)\times Z_Q \times O_{Q,\tau})$ where $O_{Q,\tau}$ is a relatively
compact neighborhood of $x$ in $X_{Q,\tau}$ such that
$p|_{O_{Q,\tau}}\colon O_{Q,\tau} \to p(O_{Q,\tau})$ is a homeomorphism;
such a $O_{Q,\tau}$ exists since $\ga_{H_{Q,\tau}}$ is neat and hence
torsion-free.  By \cite{sap1}*{Theorem ~8.1}, the identifications induced
by $\ga$ on $U$ agree with those induced by $\ga_Q$ and these are
independent of the $\overline A_Q(s)$ coordinate.  Thus the deformation
retraction $r_t$ descends to $p(U)$ and we proceed as above.

Now suppose that $\vert S_{f}\vert > 0$, pick $v \in S_{f}$ and let $S' = S \setminus\{v\}$.  We
consider $Y = \overline{X}^{RBS}$ which we write as
$\overline{X}^{RBS}_{S'}\times X_{v}$; the case  $Y =
{}_{\Q}\overline{X}^{\tau}$ is identical.  Following \cite{BS2}*{(6.8)}, for
each face $F$ of $X_{v}$ let $x_{F}$ be the barycenter of $F$ and let
$V(F)$ be the open star of $x_{F}$ in the barycentric subdivision of
$X_{v}$. The sets $V(F)$ form an open cover of $X_{v}$.  For any $\gamma
\in \Gamma$, $\gamma V(F) = V(\gamma F)$. If $F_{1} \neq F_{2}$ are two
faces with $\dim F_{1} = \dim F_{2}$, then $V(F_{1}) \cap V(F_{2}) =
\emptyset$.  It follows that
\begin{equation*}
\gamma V(F)\cap V(F) \neq \emptyset \quad \Longleftrightarrow \quad
\gamma\in\Gamma_{F}\ ,
\end{equation*}
where $\Gamma_F = \Gamma \cap \G(k_v)_F$.  It follows from
\S\ref{ssectStabilizersBuilding} that $\Gamma_F$ fixes $F$ pointwise (since
$\G(k_v)_F$ does) and is a neat  $S'$-arithmetic subgroup (since $\G(k_v)_F$ is
a compact open subgroup of $\G(k_v)$)

Let $U = \overline{X}^{RBS}_{S'}\times V(F)$ for some open face $F$ of
$X_{v}$.  Define a deformation retraction $r_t$ of $U$ onto
$\overline{X}^{RBS}_{S'}\times F$ by $r_t(w,z) = (w, tz + (1-z)r_F(z))$,
where $r_F(z)$ is the unique point in $F$ which is closest to $z\in V(F)$.
The map $r_t$ is $\Gamma_F$-equivariant since $\Gamma_{F}$ fixes $F$
pointwise and acts by isometries.  So $r_{t}$ descends to a deformation
retraction of $p(U)$ onto $(\Gamma_F\backslash
\overline{X}^{RBS}_{S'})\times F$.  The remaining hypothesis of Lemma
~\ref{lemAdmissibilityViaRetract} is satisfied since $r_0(\gamma w, \gamma
z) = r_0(\gamma w,z)$ for $\gamma \in \Gamma_F$.  Since
$\overline{X}^{RBS}_{S'}\times F \to (\Gamma_F\backslash
\overline{X}^{RBS}_{S'})\times F$ is admissible by induction, the lemma
implies that $p|_U$ is admissible.
\end{proof}

We can now see that Theorem ~\ref{thmMainArithmetic} holds if $\ga$ is neat.
According to Proposition \ref{propSimply}, $\overline{X}^{RBS}$ and ${}_{\Q}\overline{X}^{\tau}$
are simply connected. Proposition \ref{propAdmissibleNeatCase} shows that the actions of $\Gamma$
on these spaces are admissible, and  Propositions ~\ref{propDiscontinuousRBS} and ~\ref{propDiscontinuousSatake},
that they are discontinuous. Therefore Proposition ~\ref{propGrosche}\ref{itemGrosche}
applies to both spaces.

\begin{cor}
\label{corAdmissibleSubgroupNeatCase}
For any neat $S$-arithmetic subgroup $\ga$, the actions of $E\ga$ on
$\overline{X}^{RBS}$ and $E_\tau\ga$ on ${}_{\Q}\overline{X}^{\tau}$ are admissible.
\end{cor}
\begin{proof}
By Proposition ~\ref{propGammaFixedIsEGamma} the action of $\ga/E\ga$ on
$E\ga \backslash \overline{X}^{RBS}$ is free and by Proposition
~\ref{propDiscontinuousRBS} it is discontinuous.  It follows that $E\ga
\backslash \overline{X}^{RBS} \to (\ga/E\ga)\backslash (E\ga\backslash
\overline{X}^{RBS}) = \ga \backslash \overline{X}^{RBS}$ is a covering
space (in fact a regular covering space) and thus $E\ga $ acts admissibly
if and only if $\ga$ acts admissibly.
Now apply the proposition.  The case of  ${}_{\Q}\overline{X}^{\tau}$  is
treated similarly.
\end{proof}

\begin{proof}[Proof of Theorem ~\textup{\ref{thmMainArithmetic}}]
Let $\ga'\subseteq \ga$ be a normal neat subgroup of finite index.  The
idea in the general case is to factor $\overline{X}^{RBS}\to \ga\backslash
\overline{X}^{RBS}$ as
\begin{equation*}
\overline{X}^{RBS}\to E\ga'\backslash \overline{X}^{RBS} \to
(\ga/E\ga')\backslash (E\ga'\backslash \overline{X}^{RBS}) = \ga\backslash
\overline{X}^{RBS}
\end{equation*}
and apply  Proposition ~
\ref{propGrosche}\ref{itemGrosche}  to the first map and Proposition ~
\ref{propGrosche}\ref{itemArmstrong}  to the second map.

By Proposition ~\ref{propGammaFixedIsEGamma}, $\ga'_{fix,RBS} = E\ga'$ and
hence $(E\ga')_{fix,RBS} = E\ga'$. Arguing as above, Proposition~\ref{propSimply}
shows that $\overline{X}^{RBS}$ is simply connected and Proposition ~\ref{propDiscontinuousRBS}
that $\ga'$ acts discontinuously. It follows that $E\Gamma'$ acts discontinuously since $E\Gamma'$ contains all stabilizer subgroups $\Gamma'_y$ . Corollary~\ref{corAdmissibleSubgroupNeatCase} shows that
this action is admissible as well. Then Proposition~\ref{propGrosche}\ref{itemGrosche} applies and
proves that $E\ga' \backslash \overline{X}^{RBS}$ is simply connected.

We now claim that $E\ga' \backslash
\overline{X}^{RBS}$ is locally compact.  To see this, note that $E\ga'
\backslash \overline{X}^{BS}$ is locally compact since it is triangulizable
\cite{BS2}*{(6.10)}.  Furthermore the fibers of $p_1'\colon E\ga'
\backslash \overline{X}^{BS} \to E\ga' \backslash \overline{X}^{RBS}$ have
the form $\ga'_{N_P}\backslash N_P$ which are compact.  The claim follows.
Since $\ga'/E\ga'$ acts freely and $\ga/\ga'$ is finite, the action of $\ga/E\ga'$ is
proper. So we can apply Proposition ~\ref{propGrosche}\ref{itemArmstrong} to
$\ga\backslash \overline{X}^{RBS} = (\ga/E\ga' )\backslash (E\ga'
\backslash\overline{X}^{RBS})$ and find that $\pi_1(\ga\backslash
\overline{X}^{RBS}) \cong (\ga/E\ga' ) / (\ga/E\ga' )_{fix,RBS} \cong \ga /
\ga_{fix,RBS}$ as desired.  Furthermore the proof shows that the isomorphism
is induced by the natural morphism $\ga \to \pi_1(\ga\backslash
\overline{X}^{RBS})$.

A similar proof using $E_\tau\ga'$ instead of $E\ga'$ treats the case of
$\ga\backslash {}_{\Q}\overline{X}^{\tau}$; one only needs to observe
that the fibers of $p_2'\colon E_\tau\ga' \backslash \overline{X}^{RBS} \to
E_\tau\ga' \backslash {}_{\Q}\overline{X}^{\tau}$ have the form
$\ga'_{L_{P,\tau}} \backslash \overline{W}_{P,\tau}^{RBS}$ which are
compact.
\end{proof}

\section{Appendix}
\subsection{Computations of the congruence subgroup kernel}
There is an intriguing similarity between our results on the fundamental group 
and computations of the congruence subgroup kernel $C(S,\G)$.

For any nonzero ideal $\La \subseteq \OO$, set
\begin{equation*}
\ga(\La) = \{\,\gamma \in \G(k) \mid  \text{$\rho(\gamma) \in \GL_N(\OO)$, 
$ \rho(\gamma) \equiv I \pmod{\La}$}\,\}\ .
\end{equation*}
A subgroup $\ga \subset \G(k)$ is called an \emph{$S$-congruence subgroup} if it
contains $\ga(\La)$ as a subgroup of finite index for some ideal $\La
\subseteq \OO$. In its simplest form, the {\em congruence subgroup problem} asks whether
every $S$-arithmetic subgroup of $\G(k)$ is an $S$-congruence subgroup.
The congruence subgroup kernel is a quantitative measure of how close this is to being true.
We briefly outline its definition  (see \cite{SerreBourbaki}
and  \cite{Ra1}).

Define a topology ${\mathcal T}_c$ on
$\G(k)$ by taking the set of $S$-congruence subgroups to be a fundamental
system of neighborhoods of $1$.  Similarly define a topology ${\mathcal
  T}_a$ by using the set of $S$-arithmetic subgroups.  Let $\widehat G(c)$
and $\widehat G(a)$ denote the completions of $\G(k)$ in these topologies.
Since every $S$-congruence subgroup is also
$S$-arithmetic, ${\mathcal T}_a$ is in general finer than ${\mathcal T}_c$
and we have a surjective map
\begin{equation*}
\begin{CD}
\widehat G(a) @>>> \widehat G(c) \ .
\end{CD}
\end{equation*}
The kernel of this map is  called the \emph{congruence subgroup kernel} $C(S, \G)$. 

From a more general perspective, the congruence subgroup problem is the
determination of $C(S,\G)$.  The case when $C(S,\G)=1$ is equivalent to
every $S$-arithmetic subgroup being an $S$-congruence subgroup.

Assume that $k\rank \G > 0$ and $S\rank\G \ge 2$.
Under these assumptions, it can be shown that (see \citelist{\cite{Ra1}\cite{Ra2}})
\begin{equation*}
\label{eqnCongruenceKernel}
C(S,\G)\ \cong\ \varprojlim_\La
\ga(\La)/E\ga(\La)\ 
\end{equation*}
and thus, in view of \eqref{eqnFundGp},
\begin{equation}
C(S,\G) = \varprojlim\limits_{\La} \pi_{1}(\ga(\La)\backslash\oX^{RBS})\ .
\end{equation}

Now set $\ga^*(\La) = \bigcap_{\Lb\neq 0} E\ga(\La)\cdot \ga(\Lb)$ where $\Lb$ runs
over nonzero ideals of $\OO$.  Clearly
\begin{equation*}
E\ga(\La) \subseteq \ga^*(\La) \subseteq \ga(\La) 
\end{equation*}
and $E\Gamma(\mathfrak a) = E\Gamma^*(\mathfrak a)$.
By Raghunathan's Main Lemma \cite{Ra1}*{(1.17)}, for every nonzero ideal
$\La$ there exists a nonzero ideal $\La'$ such that $\ga^*(\La)\supseteq
\ga(\La')$.  Thus $\ga^*(\La)$ is the smallest $S$-congruence subgroup
containing $E\ga(\La)$. It follows that

\begin{equation*}
C(S,\G)  \cong
\varprojlim\limits_{\La} \ga^*(\La) / E\ga(\La).
\end{equation*}
Raghunathan's main theorems in \cite{Ra1} and \cite{Ra2} show that
$C(S,\G)$ is finite under the rank assumptions.
So the second limit will stabilize if we know that
\begin{equation*}
  \ga^*(\Lb)/ E\ga(\Lb) \longrightarrow
 \ga^*(\La) / E\ga(\La)
\end{equation*}
is surjective for $\Lb\subset \La$.  This too follows from Raghunathan's
Main Lemma \cite{Ra1}*{(1.17)} applied to $\Lb$ and from the definition of
$\ga^*(\La)$.  Thus 
\begin{equation}
C(S,\G)  \cong
 \ga^*(\La) / E\ga(\La) \cong \pi_1(\ga^*(\La)\backslash\overline{X}^{RBS}) \ ,
\end{equation}
for any sufficiently small nonzero ideal $\La$ of $\OO$.

It would be interesting to have an explanation for these topological interpretations of the congruence subgroup kernel.

\end{document}